\documentclass[a4paper,12pt,reqno,twoside]{amsart}

\usepackage[utf8]{inputenc} 
\usepackage[T1]{fontenc} 

\usepackage{amssymb}
\usepackage{amsmath}
\usepackage{graphicx}
\usepackage{psfrag,color}
\usepackage[colorlinks,citecolor=blue,urlcolor=blue]{hyperref}
\usepackage{cite}
\usepackage[hmargin=2cm,vmargin=2cm]{geometry}

\newcommand{\be}{\begin{eqnarray}}
\newcommand{\ee}{\end{eqnarray}}

\newcommand{\beq}{\begin{equation}}
\newcommand{\eeq}{\end{equation}}
\newcommand{\beqn}{\begin{equation*}}
\newcommand{\eeqn}{\end{equation*}}

\newcommand{\average}[1]{\langle#1\rangle}
\newcommand{\slot}{\,\cdot\,}

\providecommand{\abs}[1]{{\lvert#1\rvert}}

\def\bfSigma{\mathbf{\Sigma}}

\newtheorem{thm}{Theorem}
\newtheorem{prop}[thm]{Proposition}
\newtheorem{cor}[thm]{Corollary}
\newtheorem{lem}[thm]{Lemma}

\newtheorem{remark}[thm]{Remark}

\newcommand\cF{{\mathcal F}}

\newcommand\cH{{\mathcal H}}

\newcommand\cL{{\mathcal L}}

\newcommand\cP{{\mathcal P}}
\newcommand\cQ{{\mathcal Q}}

\newcommand\cS{{\mathcal S}}
\newcommand\cT{{\mathcal T}}

\newcommand\bC{{\mathbb C}}

\newcommand\bE{{\mathbb E}}

\newcommand\bN{{\mathbb N}}

\newcommand\bP{{\mathbb P}}

\newcommand\bR{{\mathbb R}}
\newcommand\bS{{\mathbb S}}

\newcommand\rT{{\mathrm T}}

\newcommand\rd{{\mathrm d}}

\newcommand\fm{{\mathfrak m}}

\newcommand{\ve}{\varepsilon}

\def\bfA{\mathbf{A}}
\def\bfB{\mathbf{B}}

\def\bfS{\mathbf{S}}

\def\bff{\mathbf{f}} 

\def\bft{\mathbf{t}}

\def\bfv{\mathbf{v}}

\def\bfone{\mathbf{1}}

\begin{document}

\title{A coupling approach to random circle maps expanding on the average}

\author[Mikko Stenlund]{Mikko Stenlund}
\address[Mikko Stenlund]{
Department of Mathematics and Statistics, P.O.\ Box 68, Fin-00014 University of Helsinki, Finland.}
\email{mikko.stenlund@helsinki.fi}
\urladdr{http://www.math.helsinki.fi/mathphys/mikko.html}

\author[Henri Sulku]{Henri Sulku}
\address[Henri Sulku]{
Department of Mathematics and Statistics, P.O.\ Box 68, Fin-00014 University of Helsinki, Finland.}
\email{henri.sulku@helsinki.fi}





\begin{abstract}
We study random circle maps that are expanding on the average. Uniform bounds on neither expansion nor distortion are required. We construct a coupling scheme, which leads to exponential convergence of measures (memory loss) and exponential mixing. Leveraging from the structure of the associated correlation estimates, we prove an almost sure invariance principle for vector-valued observables. The motivation for our paper is to explore these methods in a nonuniform random setting.
\end{abstract}

\maketitle


\subsection*{Acknowledgements}
This research was funded by ERC Advanced Grant MPOES. Mikko Stenlund also received financial support from the Academy of Finland. He wishes to thank Carlangelo Liverani and the University of Rome ``Tor Vergata'' for their hospitality during the preparation of this manuscript.



\medskip
\section{Introduction}

In this paper we study random compositions of the form $T_{\omega_n}\circ\dots\circ T_{\omega_1}$, where each $T_{\omega_i}$ is a~$C^2$ circle mapping with no critical points ($T_{\omega_i}'\ne0$ everywhere), drawn independently of the others from a set $\{T_{\omega_1}\,:\,{\omega_1}\in\Omega\}$ according to a probability distribution $\rd\eta(\omega_1)$. We do not place uniform bounds on expansion or distortion that would hold from one map to the next. On the contrary, individual maps are allowed to contract locally and distort their images strongly (without a bound). To compensate for such individual freedom, we impose probabilistic conditions on the occurrence of these ``bad'' maps in the sequence. In particular, we require the maps to be expanding on the average, i.e., $\int \inf|T_{\omega_1}'|\,\rd\eta(\omega_1) > 1$, with integrable distortion. The precise, somewhat stronger, assumptions are laid out in the next section. We prove statistical properties including the existence of an absolutely continuous invariant measure, exponential memory loss and mixing, as well as an almost sure invariance principle for vector-valued observables. Prior studies entailing similar models include~\cite{Ohno_1983,Pelikan_1984,Morita_1985,KhaninKifer_1996,Buzzi_1999,Buzzi_2000,Kifer_2008}. After finishing the present manuscript, the authors have also learned of the very recent works~\cite{Froyland_etal_2012, Aiminio_etal_2013} on the subject, as well as the related~\cite{Tumel_2012}.

\medskip
The motivation for the paper is twofold. First, we wish to explore the suitability of the \emph{coupling method} in the above context of nonuniform random maps. Diverting from the papers mentioned, the primary instrument in our analysis is indeed coupling. The coupling method is a soft tool for establishing statistical properties pertaining to the issues of memory loss and correlation decay. In the field of dynamical systems it has been implemented in various works such as \cite{Young_1999,BressaudLiverani,Chernov_2006,OttStenlundYoung,Stenlund_2011,StenlundYoungZhang_2012} and many others. A transparent introduction to coupling for dynamical systems (in the most elementary setup) can be found in~\cite{Sulku_2012}. As to the second motivation, a question that arises naturally is whether other limit laws hold true; we wish to investigate the possibility of proving such laws for the present class of nonuniform random dynamical systems \emph{via correlation estimates}. It was shown in~\cite{Pene_2005,ChernovMarkarian_2006,Stenlund_2010} that a central limit theorem for Sinai billiards follows from correlation bounds involving suitable classes of observables. In~\cite{Stenlund_2012} a similar approach was taken to prove an almost sure invariance principle (ASIP) for both random and non-random billiard systems. Here we show that, for our system, an ASIP follows from the established correlation estimates with little added work. Yet, the last point is subtle: it depends on the particular form of the correlation estimates, obtained for particular classes of observables. Let us be fully clear that the (averaged) theorems on the Markov chain corresponding to the random maps at issue can certainly be obtained, for example, via spectral methods. Here we present a different approach, which we hope to be of use to other authors beyond the present setup.

\bigskip
\noindent{\it Structure of the paper.} The paper is organized as follows. In Section~\ref{sec:preliminaries} we introduce the model precisely and record some mathematical preliminaries necessary for understanding the results and the proofs in the rest of the paper. In Section~\ref{sec:results} we present our main results. In the following Sections~\ref{sec:stationary}--\ref{sec:degeneracy} we prove these results in the same order as they appear in Section~\ref{sec:results}.


\medskip
\section{Preliminaries}\label{sec:preliminaries}
Let $\bS$ denote the circle obtained by identifying the endpoints of the unit interval $[0,1]$. The Lebesgue measure on $\bS^1$ is denoted by~$\fm$.

\medskip
Given $\alpha\in(0,1)$, we denote by $C^\alpha$ the set of functions $\bS\to\bR$ (or $\bS\to\bC$) that are Hölder continuous with exponent~$\alpha$. The corresponding Hölder constant is denoted by $|f|_\alpha$. We also introduce the norm
\beqn
\| f \|_\alpha = |f|_\alpha + \|f\|_\infty.
\eeqn

\medskip

To define the compositions $T_{\omega_n}\circ\dots\circ T_{\omega_1}$ of the Introduction properly,
let $(\Omega,\cF,\eta)$ be a probability space and, for each $\omega_1\in\Omega$, let the map $T_{\omega_1}:\bS\to\bS$ be $C^2$ without critical points (with additional assumptions to follow shortly). Then consider compositions of such maps drawn from the product space. We assume the map $\Omega\times\bS\to\bS:(\omega_1,x)\mapsto T_{\omega_1}x$ to be measurable, and define the quantities
\beqn
\lambda_{\omega_i} = \inf |T_{\omega_i}'|
\quad\text{and}\quad
\Delta_{\omega_i} = \left\| \frac{T_{\omega_i}''}{(T_{\omega_i}')^2} \right\|_\infty .
\eeqn
Notice that $\lambda_{\omega_i}$ measures the dilation and $\Delta_{\omega_i}$ the distortion of the map $T_{\omega_i}$.

\medskip
Expectations with respect to the ``selection distribution'' $\eta$ will often be denoted by angular brackets $\average{\slot}$. That is, for any measurable function $h:\Omega\to\bR:h(\omega_1) = h_{\omega_1}$, we write
\beqn
\langle h \rangle = \int h_{\omega_1}\,\rd\eta(\omega_1).
\eeqn

\bigskip
\noindent{\bf Standing assumption.} {\it Throughout the paper, we assume that the moment conditions
\beq\label{eq:moment_conditions}
\langle\lambda^{-2} \rangle < 1 \quad\text{and}\quad  \langle\Delta^2\rangle < \infty
\eeq
be satisfied.
}

\bigskip
\noindent In particular,~$\average{\lambda}>1$, meaning that the composed maps are expanding on the average. An individual map, on the other hand, could have regions of strong contraction,~$T_{\omega_i}'\approx 0$, as well as those of strong distortion,~$|T_{\omega_i}''|\gg (T_{\omega_i}')^2$. 

\bigskip
The sequence $(X_n)_{n\geq 0}$ with
\beqn
X_n(\omega,x) = X_n(\omega_1,\dots,\omega_n,x)= T_{\omega_n}\circ\dots\circ T_{\omega_1}(x),
\eeqn
where $\omega=(\omega_n)_{n\geq 1}\in \Omega^\bN$ and $x\in\bS$, forms a homogeneous Markov chain with state space $\bS$. (We set $X_0(\omega,x) = x$.)
The Markov operator $\cQ$ corresponding to $(X_n)_{n\ge0}$ has the expression
\beqn
\cQ f(x) = \int_\Omega f(T_{\omega_1} x)\,\rd \eta(\omega_1)
\eeqn
for any bounded measurable function $f$. Let us also define the operator $\cP$ as
\beqn
\cP g(x) = \int \cL_{\omega_1} g(x) \,\rd\eta(\omega_1), \quad g\in L^1(\fm),
\eeqn
where $\cL_{\omega_i}:L^1(\fm)\to L^1(\fm)$ stands for the transfer operator of the map $T_{\omega_i}$ associated to the Lebesgue measure~$\fm$, that is,
\beqn
\cL_{\omega_i} g(x)=\sum_{y\in T_{\omega_i}^{-1}\{x\}} \frac{g(y)}{|T_{\omega_i}'(y)|}\ .
\eeqn
As is straightforward to check, it is the dual of $\cQ$ in the sense that
\beqn
\int g\cdot \cQ f\,\rd \fm = \int \cP g\cdot f\,\rd \fm.
\eeqn

\medskip
A probability distribution $\mu$ is stationary for the Markov chain $(X_n)_{n\geq 0}$ if
\beqn
\int \cQ f\,\rd\mu = \int f\,\rd\mu
\eeqn
for all bounded measurable $f$. If $\mu$ is absolutely continuous with density $\phi$ (with respect to $\fm$), the stationarity condition reduces to
\beqn
\cP\phi = \phi.
\eeqn

\medskip

For brevity, we will write $\bE$ for the expectation with respect to the product measure~$\bP = \eta^\infty$. That is, if $h:\Omega^n\to\bR$ is measurable, then
\beqn
\bE [h] = \int h(\omega_1,\dots,\omega_n)\,\rd\eta^n(\omega_1,\dots,\omega_n) .
\eeqn
We denote by $P^\mu$ the measure induced on the path space of the Markov chain $(X_n)_{n\geq 0}$ with initial measure $\mu$. The corresponding expectation we denote $E^\mu$. That is,
\beqn
E^\mu[h(X_0,\dots,X_n)] = \int h(X_0,\dots,X_n)(\omega_1,\dots,\omega_n,x)\,\rd\eta^n(\omega_1,\dots,\omega_n)\,\rd\mu(x),
\eeqn
for any bounded measurable $h:\bS^{n+1}\to\bR$ and any $n\geq 0$.

\medskip
Finally, $\sigma$ will denote the usual left shift on $\Omega^\bN$. That is,
\beqn
(\sigma\omega)_n = \omega_{n+1}, \quad n\geq 1.
\eeqn


\medskip
\section{Results}\label{sec:results}

\medskip
The next theorem is our first result.
\begin{thm}\label{thm:stationary}
The Markov chain $(X_n)_{n\geq 0}$ admits an absolutely continuous stationary probability distribution~$\mu$ whose density $\phi$ is Lipschitz continuous and bounded away from zero.
\end{thm}

\begin{remark}
In Corollary~\ref{cor:stationary_lowerbound} we establish a ``quantitative'' lower bound on $\phi$ depending only on the ``system constants'' appearing in~\eqref{eq:moment_conditions}.
\end{remark}

From here on, $\mu$ and $\phi$ will always refer to the objects above. Once the existence of~$\phi$ has been established, it is interesting to study convergence of initial densities toward it. To that end, we first work with individual sequences~$\omega$.

\begin{thm}\label{thm:seq_bounds}
There exists such a constant $\theta\in(0,1)$ that the following holds.
Let $\alpha\in(0,1)$. For almost every $\omega$, there exists $C(\omega)>0$ such that
\beq\label{eq:seq_conv}
\|\cL_{\omega_n}\cdots \cL_{\omega_1} (\psi^1-\psi^2)\|_{L^1(\fm)} \le C(\omega)\max(\|\psi^1\|_\alpha\,,\|\psi^2\|_\alpha)\theta^{\alpha n}
\eeq
for all $n\ge 0$ and all probability densities $\psi^1,\psi^2\in C^\alpha$. Moreover, given a probability distribution~$\rd\nu = \psi\,\rd\fm$ with $\psi\in C^\alpha$,
\beq\label{eq:seq_mix}
\left|\int f\cdot g\circ T_{\omega_n}\circ\cdots\circ T_{\omega_1}\,\rd\nu - \int f\,\rd\nu \int g\circ T_{\omega_n}\circ\cdots\circ T_{\omega_1}\,\rd\nu \right| \le C(\omega)\|\psi\|_\alpha \|f\|_\alpha\|g\|_\infty \theta^{\alpha n}
\eeq
for all $n\ge 0$, and all complex-valued functions $f\in C^\alpha$ and $g\in L^\infty$.
\end{thm}

Here~\eqref{eq:seq_conv} states that, for typical sequences~$\omega$, the $L^1$-distance between the push-forwards of two Hölder continuous densities tends to zero exponentially.
The bound in~\eqref{eq:seq_mix} states that, with respect to any probability measure having a Hölder continous density, the random variables~$f$ and~$g\circ T_{\omega_n}\circ\cdots\circ T_{\omega_1}$ become asymptotically decorrelated at an exponential rate. The method of proof we use is coupling. Theorem~\ref{thm:seq_bounds} can also be obtained by different means, namely that of thermodynamic formalism and Hilbert projective cones; see~\cite{KhaninKifer_1996,Kifer_2008} and, for a piecewise smooth case,~\cite{Buzzi_1999}.

\medskip
Once the sequence-wise bounds have been obtained, related results can be established for the Markov chain $(X_n)_{n\ge 0}$:

\begin{thm}\label{thm:bounds}
There exist a constant $\theta\in(0,1)$ and, for any $\alpha\in(0,1)$, a constant $C>0$ such that
\beq\label{eq:conv}
\|\cP^n \psi-\phi\|_{L^1(\fm)} \le C\|\psi\|_\alpha\theta^{\alpha n}
\eeq
for all $n\ge 0$ and all probability densities $\psi\in C^\alpha$. Moreover, 
\beq\label{eq:mix}
\left|\int f\cdot \cQ^n g\,\rd\mu -\int f\,\rd\mu\int g\,\rd\mu \right| \le C \|f\|_\alpha\|g\|_\infty\theta^{\alpha n}
\eeq
for all $n\ge 0$, and all complex-valued functions $f\in C^\alpha$ and $g\in L^\infty$.
\end{thm}
By~\eqref{eq:conv}, the push-forwards of Hölder continuous initial densities converge in $L^1$ to the Lipschitz continuous invariant density at an exponential rate, while pair correlations with respect to the stationary distribution decay exponentially by~\eqref{eq:mix}.
In the present formulation, Theorem~\ref{thm:bounds} does strictly speaking not follow from Theorem~\ref{thm:seq_bounds}, because we do not claim that $C(\omega)$ has finite expectation. Rather, we will prove the two results in parallel, as consequences of common intermediate bounds.

\medskip
By Theorem~\ref{thm:bounds}, the measure~$\mu$ is ergodic. It is standard that distinct ergodic measures are mutually singular. Since~$\mu$ is equivalent to $\fm$ by Theorem~\ref{thm:stationary}, we get the following corollary:
\begin{cor}
The measure~$\mu$ is the unique absolutely continuous ergodic measure.
\end{cor}

\medskip
Given sufficient information on the convergence of measures, it becomes natural to ask about the statistical properties of the limit distribution.
Indeed, the pair correlation bound in~\eqref{eq:mix} is key in our proof of the probabilistic limit theorem below. The investigation of the coupling technique aside (see Introduction), it is the main result of our paper. To the best of our knowledge, such a result has not appeared in the literature.

\pagebreak
\begin{thm}\label{thm:VASIP}
Fix a positive integer $d$. Let $\bff:\bS\to\bR^d$ be Hölder continuous with $\int \bff\,\rd\mu = 0$, and denote briefly
\beq\label{eq:A_n}
\bfA_n = \bff\circ X_n \ .
\eeq
There exists such a symmetric, semi-positive-definite, $d\times d$ matrix~$\bfSigma^2$ that the following  hold:
\begin{enumerate}
\item The matrix $\bfSigma^2$ is the limit covariance of $\frac1{\sqrt n}\sum_{k=0}^{n-1}\bfA_k$. That is, 
\beqn
\lim_{n\to\infty} \frac1n\,E^\mu\!\left(\sum_{k=0}^{n-1}\bfA_k \otimes \sum_{k=0}^{n-1}\bfA_k\right)=\bfSigma^2\ .
\eeqn
\medskip
\item The random variables $\frac1{\sqrt n}\sum_{k=0}^{n-1}\bfA_k$ converge in distribution, as $n\to\infty$, to a centered $\bR^d$-valued normal random variable with covariance $\bfSigma^2$.
\medskip
\item Given any $\lambda>\frac14$, there exists a probability space together with two $\bR^d$-valued processes $(\bfA^*_n)_{n\geq 0}$ and $(\bfB_n)_{n\geq 0}$ on it, for which the following statements are true:
\medskip
\begin{enumerate}
\item $(\bfA_n)_{n\geq 0}$ and $(\bfA^*_n)_{n\geq 0}$ have the same distribution.
\medskip
\item The random variables $\bfB_n$, $n\geq 0$, are independent, centered, and normally distributed with covariance $\bfSigma^2$.
\medskip
\item Almost surely, $\bigl|\sum_{k=0}^{n-1}\bfA^*_k - \sum_{k=0}^{n-1}\bfB_k\bigr| = o(n^\lambda)$.  
\end{enumerate}
\end{enumerate}

\end{thm}

\medskip
Item~(2) of the theorem is called an averaged (or annealed) central limit theorem and
item~(3) a vector-valued almost sure invariance principle with covariance~$\bfSigma^2$ and error exponent~$\lambda$. The ``almost surely'' in item (c) refers to the probability space on which the processes~$(\bfA^*_n)_{n\geq 0}$ and~$(\bfB_n)_{n\geq 0}$ are defined. Note that $\sum_{k=0}^{n-1}\bfB_k$ can be interpreted as the location of an $\bR^d$-valued Brownian motion at time $n$. The almost sure invariance principle implies several other limit results, which we do not list here; see~\cite{Strassen_1964,Billingsley_convergence,PhilippStout_1975,LaceyPhilipp_1990}. 

\medskip
A standard computation in item~(1) yields the formula
\beqn
\bfSigma^2 =
\int \bff \otimes \bff \, \rd \mu
+  \sum_{m=1}^\infty \int \bigl(\bff \otimes \cQ^m\bff + \cQ^m\bff \otimes \bff \bigr) \, \rd \mu \ .
\eeqn
The question arises whether this matrix is non-degenerate.
\begin{lem}\label{lem:degeneracy}
Consider a nonzero vector $\bfv\in\bR^d$. The matrix $\bfSigma^2$ is degenerate in the direction~$\bfv$, i.e., $\bfv^\rT\bfSigma^2\bfv = 0$, if and only if there exists a Hölder continuous function $g:\bS\to\bR$ satisfying
\beqn
\bfv^\rT\bff(x) = g(x) - g(T_{\omega_1} x)
\eeqn
for all $x$ and almost all $\omega_1$. (Here the superscript $\rT$ denotes transposition.)
\end{lem}
The preceding lemma places a serious obstruction to degeneracy. In particular, up to a negligible set of $\omega$'s,
\beqn
\sum_{k=0}^{p-1} \bfv^\rT\bff(T_{\omega_k}\circ\dots\circ T_{\omega_1}(x)) = 0
\eeqn
whenever $T_{\omega_p}\circ\dots\circ T_{\omega_1}(x) = x$ (periodic trajectory). Hence, having a degenerate covariance matrix $\bfSigma^2$ amounts to a very exceptional choice of~$\bff$.


\medskip
\section{Proof of Theorem~\ref{thm:stationary}}\label{sec:stationary}

The strategy of proving Theorem~\ref{thm:stationary} is to find $\phi$ as an accumulation point of $\bigl(n^{-1}\sum_{k=0}^{n-1}\cP^k\bfone\bigr)_{n\ge 1}$ by showing that $\|\cP^n \bfone\|_{C^1}$ is uniformly bounded for $n\ge 0$.

\medskip

Given a sequence $(\omega_i)_{i\ge 1}$, denote 
\beq\label{eq:SR}
S_n = \prod_{i=1}^n  \lambda_{\omega_i}^{-1}\quad\text{and}\quad R_n = \sum_{i=1}^n \Delta_{\omega_i} \prod_{j=i+1}^n \lambda_{\omega_j}^{-1}.
\eeq

\begin{lem}\label{lem:L_prod_diff}
For any $C^1$-function $\psi$, the bound
\beqn
 |(\cL_{\omega_n}\dots\cL_{\omega_1} \psi)'| \le S_n \cdot \cL_{\omega_n}\cdots\cL_{\omega_1} |\psi'| + R_n\cdot \cL_{\omega_n}\cdots\cL_{\omega_1} |\psi|
 \eeqn
holds.
\end{lem}

\begin{proof}
The straightforward bound
\beqn
|(\cL_{\omega_1}\psi)'| \leq \lambda_{\omega_1}^{-1}\cL_{\omega_1}|\psi'|+\Delta_{\omega_1}\cL_{\omega_1}|\psi|, \qquad \omega_1\in\Omega,
\eeqn
holds for a $C^1$-function $\psi$. Iterating this bound yields the claim.
\end{proof} 
First of all, Lemma~\ref{lem:L_prod_diff} implies
\beqn
 \|(\cL_{\omega_n}\dots\cL_{\omega_1} \psi)'\|_{L^1(\fm)} \le S_n\|\psi'\|_{L^1(\fm)} + R_n\|\psi\|_{L^1(\fm)},
\eeqn
because $\cL_{\omega_i}$ is a contraction in $L^1(\fm)$. Setting $\psi = \bfone$ yields
$
 \|(\cL_{\omega_n}\dots\cL_{\omega_1} \bfone)'\|_{L^1(\fm)} \le R_n.
$
Thus,
\beq\label{eq:L_prod_sup}
\|\cL_{\omega_n}\cdots\cL_{\omega_1} \bfone\|_\infty \le 1+\|(\cL_{\omega_n}\cdots\cL_{\omega_1} \bfone)'\|_{L^1(\fm)} \le 1+R_n,
\eeq
because $\cL_{\omega_n}\cdots\cL_{\omega_1} \bfone$ is a $C^1$ probability density. Another consequence of Lemma~\ref{lem:L_prod_diff} is
\beqn
\|(\cL_{\omega_n}\cdots\cL_{\omega_1} \bfone)'\|_\infty \le R_n \|\cL_{\omega_n}\cdots\cL_{\omega_1} \bfone\|_\infty \le R_n(1+R_n).
\eeqn
In particular,
\beqn
\|\cP^n \bfone\|_{C^1} \le \bE[(1+R_n)^2].
\eeqn

\begin{lem}\label{lem:R^2_bound}
There exists such a constant $C_R>0$ that
\beqn
\sup_{n\ge 1} \bE[(1+R_n)^2] \le C_R.
\eeqn
\end{lem}
\begin{proof}
By Jensen's inequality, it is enough to check that $\bE[R_n^2]$ is uniformly bounded.
But
\beqn
\begin{split}
R_n^2 &= \sum_{i=1}^n\sum_{\ell=1}^n \Delta_{\omega_i}\Delta_{\omega_\ell} \prod_{j=i+1}^n \lambda_{\omega_j}^{-1} \prod_{k=\ell+1}^n \lambda_{\omega_k}^{-1}
\\
& = \sum_{i=1}^n \Delta_{\omega_i}^2 \prod_{j=i+1}^n \lambda_{\omega_j}^{-2} + 2\sum_{1\le i<\ell\le n} \Delta_{\omega_i}\Delta_{\omega_\ell} \prod_{j=i+1}^n \lambda_{\omega_j}^{-1} \prod_{k=\ell+1}^n \lambda_{\omega_k}^{-1}
\\
& = \sum_{i=1}^n \Delta_{\omega_i}^2 \prod_{j=i+1}^n \lambda_{\omega_j}^{-2} + 2\sum_{1\le i<\ell\le n} \Delta_{\omega_i}\Delta_{\omega_\ell}\lambda_{\omega_\ell}^{-1} \prod_{j=i+1}^{\ell-1} \lambda_{\omega_j}^{-1} \prod_{k=\ell+1}^n \lambda_{\omega_k}^{-2}.
\end{split}
\eeqn
Therefore,
\beqn
\bE[R_n^2] = \langle\Delta^2\rangle\sum_{i=1}^n   \langle\lambda^{-2}\rangle^{n-i} + 2\langle\Delta\rangle\,\langle\lambda^{-1}\Delta\rangle \sum_{1\le i<\ell\le n}  \langle\lambda^{-1}\rangle^{\ell-1-i} \, \langle\lambda^{-2}\rangle^{n-\ell}.
\eeqn
Here
\beqn
\langle\lambda^{-1}\rangle \le \langle\lambda^{-2}\rangle^{1/2},\quad \langle\Delta\rangle \le \langle\Delta^2\rangle^{1/2} \quad\text{and}\quad  \langle\lambda^{-1}\Delta\rangle \le \langle\lambda^{-2}\rangle^{1/2}  \langle\Delta^2\rangle^{1/2}
\eeqn
by Jensen's and Hölder's inequalities. Thus, by~\eqref{eq:moment_conditions},
\beqn
\sum_{1\le i<\ell\le n}  \langle\lambda^{-1}\rangle^{\ell-1-i} \, \langle\lambda^{-2}\rangle^{n-\ell} 
= \sum_{\ell=2}^n \left(\sum_{i=1}^{\ell-1} \langle\lambda^{-1}\rangle^{\ell-1-i} \right) \langle\lambda^{-2}\rangle^{n-\ell} 
\le
 \frac{1}{1-\langle\lambda^{-1}\rangle} \frac{1}{1-\langle\lambda^{-2}\rangle} < \infty.
\eeqn
This proves the lemma.
\end{proof}

It is standard that the existence of a Lipschitz continuous stationary distribution as an accumulation point of the sequence $(n^{-1}\sum_{k=0}^{n-1}\cP^k\bfone)_{n\ge 1}$ follows by a compactness argument (see, e.g.,~\cite{Sulku_2012}). The distribution is strictly positive. To see this, first observe that $\phi>0$ on an arc $I\subset\bS$. Also, there exists a $\bar\lambda>1$ such that $\eta\bigl(\omega_1\in\Omega\,:\,\lambda_{\omega_1}\ge \bar\lambda\bigr)>0$. Thus, we have~$\cP^n(\phi |_I)>0$ for some sufficiently large $n$. Now $\phi = \cP^n(\phi) \ge \cP^n(\phi |_I)>0$. 

\medskip
The proof of Theorem~\ref{thm:stationary} is now complete.
\qed


\medskip
\section{Proofs of Theorems~\ref{thm:seq_bounds} and~\ref{thm:bounds}}

\subsection{Regularity of push-forward densities}
The following distortion estimate is standard. It will be needed for controlling the regularity of the push-forward distributions under the dynamics.
\begin{lem}
Let $n\in\bN$ be arbitrary. For any $x,y\in\bS$,
\beq
e^{-R_n d(x,y)}\leq \frac{(T_{\omega_n}\circ\cdots\circ T_{\omega_1})'((T_{\omega_n}\circ\cdots\circ T_{\omega_1})^{-1}_ix)}{(T_{\omega_n}\circ\cdots\circ T_{\omega_1})'((T_{\omega_n}\circ\cdots\circ T_{\omega_1})^{-1}_iy)} \leq e^{R_n d(x,y)},
\eeq
Here $R_n$ is as defined earlier and $(T_{\omega_n}\circ\cdots\circ T_{\omega_1})^{-1}_i$ is the $i$th branch of the inverse of $T_{\omega_n}\circ\cdots \circ T_{\omega_1}$ on a given arc $J\subset \bS$ of length $|J|\leq \frac12$ containing both $x$ and $y$.
\end{lem}

\begin{proof}
For brevity, let $x_{-n+i}$ and $y_{-n+i}$ denote the preimages of $x$ and $y$, respectively, along the same branch of the inverse of $T_{\omega_n}\circ\cdots\circ T_{\omega_{i+1}}$. Note that
\beqn
(T_{\omega_n}\circ\cdots\circ T_{\omega_1})'(x_{-n}) = \prod_{i=1}^n T_{\omega_i}'\circ S_{\omega_i} (x_{-n+i}),
\eeqn
where $S_{\omega_i}$ stands for an appropriate inverse branch of $T_{\omega_i}$. Hence,
\begin{align*}
& \log \frac{(T_{\omega_n}\circ\cdots\circ T_{\omega_1})'(x_{-n})}{(T_{\omega_n}\circ\cdots\circ T_{\omega_1})'(y_{-n})} \leq |\log((T_{\omega_n}\circ\cdots\circ T_{\omega_1})'(x_{-n}))-\log((T_{\omega_n}\circ\cdots\circ T_{\omega_1})'(y_{-n}))| \\
&
\qquad\qquad\qquad \leq \sum^{n}_{i=1}|\log T_{\omega_i}'\circ S_{\omega_i} (x_{-n+i})-\log T_{\omega_i}'\circ S_{\omega_i} (y_{-n+i})| \\
&
\qquad\qquad\qquad \leq \sum^{n}_{i=1} \|(\log T_{\omega_i}'\circ S_{\omega_i})' \|_\infty \,d(x_{-n+i},y_{-n+i}) 
\leq \sum^{n}_{i=1}\Delta_{\omega_i} \, d(x_{-n+i},y_{-n+i}) \\
&
\qquad\qquad\qquad \leq \sum^{n}_{i=1}\Delta_{\omega_i} \prod_{j=i+1}^n\lambda_{\omega_j}^{-1} \cdot d(x,y)
= R_n d(x,y).
\end{align*}
A similar estimate is obtained by interchanging $x$ and $y$, which proves the claim. 
\end{proof}

\begin{prop}
\label{prop:Hoelder_estimate}
Suppose $\psi$ is a strictly positive probability density and that $\log\psi \in C^\alpha$.
Then $\cL_{\omega_n}\cdots\cL_{\omega_1}\psi$ inherits these properties for every $n\in\bN$ and 
\beqn
\abs{\log\cL_{\omega_n}\cdots\cL_{\omega_1}\psi}_\alpha \leq S_{n}^\alpha \abs{\log\psi}_\alpha + R_n,
\eeqn
where $R_n$ and $S_n$ have been defined earlier.
\end{prop}

\begin{proof}
Let $J\subset \bS$ be an arc with $|J|\leq \frac12$.
Given an initial probability density $\psi$, we introduce the notation
\beqn
\psi_{n,i}(x) = \frac{\psi((T_{\omega_n}\circ\cdots\circ T_{\omega_1})^{-1}_i x)}{(T_{\omega_n}\circ\cdots\circ T_{\omega_1})'((T_{\omega_n}\circ\cdots\circ T_{\omega_1})^{-1}_i x)} \ ,\qquad x\in J,
\eeqn
Then
\beqn
\cL_{\omega_n}\cdots\cL_{\omega_1}\psi(x) = \sum_{i=1}^{w} \psi_{n,i}(x),
\eeqn
where $w$ is the number of inverse branches of $T_{\omega_n}\circ\cdots \circ T_{\omega_1}$ on $J$. Next, let $x,y\in\bS$ be arbitrary. Without loss of generality, we may assume both points belong to $J$. Therefore,
\beqn
\begin{split}
\left|\log\frac{\psi_{n,i}(x)}{\psi_{n,i}(y)}\right| 
&
\leq \left|\log\frac{\psi((T_{\omega_n}\circ\cdots\circ T_{\omega_1})^{-1}_i x)}{\psi((T_{\omega_n}\circ\cdots\circ T_{\omega_1})^{-1}_i y)}\right| + \left|\log\frac{(T_{\omega_n}\circ\cdots\circ T_{\omega_1})'((T_{\omega_n}\circ\cdots\circ T_{\omega_1})^{-1}_i y )}{(T_{\omega_n}\circ\cdots\circ T_{\omega_1})'((T_{\omega_n}\circ\cdots\circ T_{\omega_1})^{-1}_i x)}\right|
\\
&
\leq |\log\psi|_\alpha \,d((T_{\omega_n}\circ\cdots\circ T_{\omega_1})^{-1}_i x, (T_{\omega_n}\circ\cdots\circ T_{\omega_1})^{-1}_i y)^\alpha + R_n d(x,y)
\\
&
\leq (S_{n}^\alpha |\log\psi|_\alpha \,+ R_n) d(x,y)^\alpha.
\end{split}
\eeqn
For brevity, denote $B_n=(S_{n}^\alpha |\log\psi|_\alpha \,+ R_n) d(x,y)^\alpha$. Then
\beqn
e^{-B_n}\psi_{n,i}(y) \leq \psi_{n,i}(x)\leq e^{B_n}\psi_{n,i}(y).
\eeqn
Summing over $i$, we get
\beqn
e^{-B_n}\cL_{\omega_n}\cdots\cL_{\omega_1}\psi(y) \leq \cL_{\omega_n}\cdots\cL_{\omega_1}\psi(x)\leq e^{B_n}\cL_{\omega_n}\cdots\cL_{\omega_1}\psi(y).
\eeqn
Taking logarithms yields the desired bound.
\end{proof}

\medskip
Before proceeding, we prove the lower bound alluded to below Theorem~\ref{thm:stationary} on the stationary density~$\phi$ in terms of system constants:
\begin{cor}\label{cor:stationary_lowerbound}
There exists a constant $c>0$, depending only on the moments appearing in~\eqref{eq:moment_conditions}, for which
\beqn
\inf \phi \ge c.
\eeqn
\end{cor}
\begin{proof}
Recall that $\phi$ is an accumulation point of the sequence $(n^{-1}\sum_{k=0}^{n-1}\cP^k\bfone)_{n\ge 1}$. From Proposition~\ref{prop:Hoelder_estimate} we get $\log{\cL_{\omega_n}\cdots\cL_{\omega_1}\bfone} \ge -R_n$ for all $n\ge 0$. Applying Jensen's inequality,
\beqn
\cP^n\bfone = \bE(\cL_{\omega_n}\cdots\cL_{\omega_1}\bfone) \ge \bE\bigl(e^{-R_n}\bigr) \ge e^{-\bE(R_n)}.
\eeqn
Lemma~\ref{lem:R^2_bound}, together with another application of Jensen's inequality, finishes the proof.
\end{proof}

\medskip
\subsection{Coupling argument}
We are now ready to explain the coupling step. In what follows, we assume that $\alpha>0$ has been fixed once and for all.

\medskip
We introduce the notation 
\beqn
\cH_K = \{\psi:X\to\bR\text{ a probability density},\, \psi>0,\, |\log\psi|_\alpha \leq K\},
\eeqn
with $K > 0$. The following lemma will turn out useful.

\begin{lem}\label{lem:aK}
Fix any $K > 0$ and set
\begin{align}
\kappa &= \frac{1}{2}\exp(-K) \\
K' &= \exp\left(4K\right).
\end{align}
Then
\begin{enumerate}
\item $\psi \geq 2\kappa > 0$ holds for every $\psi\in\cH_{K}$.
\item $\tilde{\psi} := (\psi-\kappa)/(1-\kappa)\in\cH_{K'}$ for all $\psi\in\cH_{K}$. 
\end{enumerate}
\end{lem}

\begin{proof}
We have the elementary bounds (see \cite{Sulku_2012})
\beqn
\exp(-|\log\psi|_\alpha) \leq \psi(x) \leq \exp(|\log\psi|_\alpha)
\eeqn
and
\beqn
|\psi|_\alpha \leq |\log\psi|_\alpha\exp\left(|\log\psi|_\alpha\right)
\eeqn
for probability densities $\psi$. Thus, for $\psi\in\cH_K$,
\beqn
\psi(x) \geq \exp(-|\log\psi|_\alpha) \geq \exp(-K) = 2\kappa.
\eeqn
Therefore,
\begin{align*}
\left|\log\!\left(\frac{\psi(x)-\kappa}{1-\kappa}\right)-\log\!\left(\frac{\psi(y)-\kappa}{1-\kappa}\right)\right|
&\le \sup_x\frac{1}{\psi(x)-\kappa}|\psi(x)-\psi(y)|
\le \frac{1}{\kappa}|\psi|_\alpha \,d(x,y)^\alpha 
\\
&\leq 2K\exp\left(2K\right)d(x,y)^\alpha
\leq \exp\left(4K\right)d(x,y)^\alpha,
\end{align*}
which proves the lemma.
\end{proof}

From here on, we will assume that $K>0$ is fixed once and for all. (The value of $K$ will be determined later.) This also fixes $\kappa$ and $K'$. 

\medskip
Given a $K''>0$ and a sequence $\omega$, we say that the 

\bigskip

\noindent
{\it Coupling condition $C(K'',\omega,n)$ is satisfied if}
\begin{align} \label{eq:coupling_cond}
S_n^\alpha K''  +
R_n  \leq K.
\end{align}

\bigskip
\noindent This definition is natural, because Proposition~\ref{prop:Hoelder_estimate} implies that 
\beqn
\cL_{\omega_n}\cdots\cL_{\omega_1}\psi\in\cH_{K} \quad \forall\, \psi\in\cH_{K''} \quad \text{if $C(K'',\omega,n)$ is satisfied}.
\eeqn
Let $\psi^1$ and $\psi^2$ be arbitrary densities in $\cH_{K''}$ and suppose $C(K'',\omega,n)$ is satisfied for some $n$. Then, by the above observation and by Lemma~\ref{lem:aK},
\beqn
\psi^i_n \equiv \cL_{\omega_n}\cdots\cL_{\omega_1}\psi^i = \kappa+(1-\kappa)\tilde{\psi}^i_n,
\eeqn
where 
\beq\label{eq:psi^i_n}
\tilde\psi^i_n = (\psi^i_n-\kappa)/(1-\kappa)\in\cH_{K'}.
\eeq
Thus,
\beqn
\|\psi^1_n-\psi^2_n\|_{L^1(\fm)} \leq (1-\kappa)\|\tilde{\psi}^1_n-\tilde{\psi}^2_n\|_{L^1(\fm)}.
\eeqn
In other words, the coupling condition allowed us to ``couple'' a $\kappa$-fraction of the $n$-step push-forwards of the densities $\psi^1$ and $\psi^2$. Obviously, the procedure can be continued inductively, treating $\tilde\psi^i_n$ as the initial densities: since~\eqref{eq:psi^i_n} holds, we can couple a $\kappa$-fraction of the $\tau$-step push-forwards $\cL_{\omega_{n+\tau}}\cdots\cL_{\omega_{n+1}}\tilde\psi^1_n$ and $\cL_{\omega_{n+\tau}}\cdots\cL_{\omega_{n+1}}\tilde\psi^2_n$ assuming that $C(K',\sigma^n\omega,\tau)$ holds, and again the ``normalized remainder densities'' are in $\cH_{K'}$ by~\eqref{eq:psi^i_n}. 

Let us formalize the above procedure.
Given $K''>0$ and $\omega$, define 
\beqn
\tau_0(\omega)=0
\quad
\text{and}
\quad
\tau_1(\omega) = \inf\{n\ge 0 \,:\, \text{$C(K'',\omega,n)$ is satisfied}\},
\eeqn
and
\beqn
\tau_k(\omega) = \inf\{n\ge 0 \,:\, \text{$C(K',\sigma^{\tau_{k-1}(\omega)}\omega,n)$ is satisfied}\}
\eeqn
for all $k\ge 1$. We use here the convention that the infimum of the empty set is $\infty$.
Next, set
\beqn
n_0(\omega) = 0 
\quad
\text{and}
\quad
n_k(\omega) = \sum_{j=1}^k \tau_j(\omega),
\eeqn
for $k\ge 1$.
Now $n_k\in \bN\cup\{\infty\}$ is the time at which the $k$th coupling will occur and $\tau_k\in\bN\cup\{\infty\}$ is the $k$th inter-coupling time. (Both depend on $K''$ and $\omega$, but we suppress this from the notation.) In particular, using the above coupling argument in combination with the $L^1(\fm)$-contractivity of each $\cL_{\omega_i}$ and the obvious fact that $\|\psi^1-\psi^2\|_{L^1(\fm)}\le 2$, we see that
\beqn
\|\psi^1_n-\psi^2_n\|_{L^1(\fm)} \leq 2(1-\kappa)^k\qquad \forall \, n\ge n_k
\eeqn
holds for all pairs $\psi^1,\psi^2\in\cH_{K''}$. Alternatively, writing
\beqn
N_n(\omega) = \max\{k\ge 0 \,:\, n_k(\omega) \le n\}
\eeqn
for the number of couplings by time $n$ for the sequence $\omega$,
\beqn
\|\psi^1_n-\psi^2_n\|_{L^1(\fm)} \leq 2(1-\kappa)^{N_n} \qquad \forall \, n\ge 0 .
\eeqn

In brief, the $L^1$-distance between the densities converges exponentially to zero as a function of the number of couplings that has occurred. To make use of this, it is necessary to study the statistical properties of $N_n$.

\medskip
\subsection{Coupling time analysis}
In this section we analyze the tail behavior of the inter-coupling times~$\tau_k$, and subsequently obtain crucial information about the distribution of~$N_n$. The task will boil down to studying a pair of random difference equations.

\medskip
For notational simplicity, let us write
\beqn
A_n = \lambda_{\omega_n}^{-1}
\quad\text{and}\quad
B_n = \Delta_{\omega_n}.
\eeqn
Given $\alpha\in(0,1)$, Proposition~\ref{prop:Hoelder_estimate} states that
\beqn
\abs{\log\cL_{\omega_n}\cdots\cL_{\omega_1}\psi}_\alpha \leq S_{n}^\alpha \abs{\log\psi}_\alpha + R_n,
\eeqn
where, according to~\eqref{eq:SR},
\beqn
S_n = \prod_{i=1}^n  A_i \quad\text{and}\quad R_n = \sum_{i=1}^n B_i \prod_{j=i+1}^n A_j.
\eeqn
Starting with $\abs{\log\psi}_\alpha = \xi$, we can perform a coupling when $S_n^\alpha\xi+R_n \le K$; see~\eqref{eq:coupling_cond}. Here~$\xi>0$ is an arbitrary initial condition and $K>0$ a large (non-random) constant to be fixed later.

\medskip
Note that $R_n$ and $\tilde S_n = S_n^\alpha\xi$ satisfy the random difference equations
\beqn
\begin{split}
R_n & = A_n R_{n-1} + B_n,\quad n\ge 1,
\end{split}
\eeqn
and
\beqn
\begin{split}
\tilde S_n & = A_n^\alpha \tilde S_{n-1},\quad n\ge 1.
\end{split}
\eeqn
Our objective is to control the random time when the coupling condition
\beqn
Z_n = R_n + \tilde S_n \le K
\eeqn
is first satisfied, given the initial condition 
\beqn
(R_0,\tilde S_0) = (0,\xi).
\eeqn
Indeed,
\beqn
\abs{\log\cL_{\omega_n}\cdots\cL_{\omega_1}\psi}_\alpha\le Z_n
\eeqn
for all $n\ge 0$. This objective is complicated by the fact that $\tilde S_n$ and $R_n$ are \emph{not} independent random variables, and because  --- unlike $(\tilde S_n)_{n\ge 0}$ and $(R_n)_{n\ge 0}$ separately --- the sequence $(Z_n)_{n\ge 0}$ of their sums does not satisfy a simple recursion relation.

To remedy the above situation, we begin with the observation that
\beqn
Z_n \le R_n + \max(\tilde S_n,1) \le R_n + \max(\tilde S_n^{1/\alpha},1) \le R_n + \tilde S_n^{1/\alpha}+1.
\eeqn
Since $\tilde S_n^{1/\alpha} = A_n \tilde S_{n-1}^{1/\alpha}$, the sums
\beqn
L_n = R_n + \tilde S_n^{1/\alpha}
\eeqn
satisfy the random difference equation
\beq\label{eq:RDE_L}
L_n = A_n L_{n-1} + B_n, \quad n\ge 1,
\eeq
with the initial condition
\beqn
L_0 = \xi^{1/\alpha}.
\eeqn
This is of interest because~\eqref{eq:RDE_L} is ``simple'' and because of the dominating property
\beqn
Z_n \le L_n + 1, \quad n\ge 0.
\eeqn
The coupling condition is therefore certainly satisfied if
\beqn
L_n \le K-1.
\eeqn
Thus, let
\beqn
T = \inf\{k\ge 0 \;:\; L_k\le K-1\}
\eeqn
be the first time the Markov chain~$(L_n)_{n\ge 0}$ dips below level~$K-1$. The utility of $T$ to our proof lies in the fact that it dominates the inter-coupling times $\tau_k$ when $\xi$ is chosen properly.

For the following, note that the assumptions in~\eqref{eq:moment_conditions} imply
\beqn
\average{A_1}<1
\quad\text{and}\quad
\average{B_1}<\infty. 
\eeqn
\begin{prop}\label{prop:coupling_time_tail}
Fix any
\beqn
K >  \frac{\average{B_1}}{1-\average{A_1}}+1.
\eeqn
Then
\beqn
q = \average{A_1}+\frac{\average{B_1}}{K-1} < 1.
\eeqn
Starting the Markov chain $(L_n)_{n\ge 0}$ at an arbitrary level $L_0 = \ell>K-1$,
\beqn
P^\ell(T> n) \le \frac{\ell q^n}{K-1}
\eeqn
for all $n\ge 0$. (Here $P^\ell$ is the path measure of $(L_n)_{n\ge 0}$ starting at $\ell$.)
\end{prop}
We point out that the number $\frac{\average{B_1}}{1-\average{A_1}}$ appearing in the lemma above is the expected value of the stationary limit distribution of the chain~$(L_n)_{n\ge 0}$; see~\cite{Vervaat_1979}.

\begin{proof}
The key idea of the proof is to dominate the chain $(L_n)_{n\ge 0}$ with another chain $(\tilde L_n)_{n\ge 0}$ whose value decays below level $K-1$ quickly.
Note that we can rewrite~\eqref{eq:RDE_L} as
\beqn
L_m  = U_m L_{m-1}
\eeqn
where
\beqn
U_m = A_m  + \frac{B_m}{L_{m-1}}.
\eeqn
Given $L_0 = \ell>K-1$, we have $T\ge 1$ and $L_{m-1}>K-1$ for all $m\in[1,T]$. Therefore,
\beqn
U_m < A_m  + \frac{B_m}{K-1} = V_m
\eeqn
and
\beqn
L_m \le V_m L_{m-1}, \quad 1\le m\le T.
\eeqn
Defining a Markov chain $(\tilde L_m)_{m\ge 0}$ such that $\tilde L_0 = \ell$ and 
\beqn
\tilde L_m = V_m \tilde L_{m-1}, \quad m\ge 0, 
\eeqn
we have
\beqn
L_m \le \tilde L_m, \quad 0\le m\le T.
\eeqn
That is, $\tilde L_m$ dominates $L_m$ for as long as the chain $(L_m)_{m\ge 0}$ remains above level~$K-1$. In particular, $T> n \;\Rightarrow\; \min_{0\le m\le n}L_m >K-1 \;\Rightarrow\; \tilde L_n>K-1$, so that
\beqn
\begin{split}
P^\ell(T> n) & \le P^\ell(\tilde L_n > K-1).
\end{split}
\eeqn
Since $\tilde L_n = \ell\prod_{m=1}^n V_m$, Markov's inequality now yields
\beqn
\begin{split}
P^\ell(\tilde L_n > K-1) < \frac{1}{K-1} E^\ell(\tilde L_n) = \frac{\ell}{K-1} (E(V_1))^n = \frac{\ell q^n}{K-1},
\end{split}
\eeqn
where
\beqn
q = \average{A_1}+\frac{\average{B_1}}{K-1} < 1
\eeqn
as we assume that $K>\frac{\average{B_1}}{1-\average{A_1}}+1$. 
\end{proof}

\medskip
\begin{prop}\label{prop:N_n}
Let $K>0$ be as in Proposition~\ref{prop:coupling_time_tail}, and $\alpha\in(0,1)$ and $K''>0$ be given.
There exist such constants $t\in(0,1)$, $\vartheta\in(0,1)$ independent of $\alpha$, and $D>0$ that
\beqn
\bP(N_n < [t\alpha n])  \le  D \vartheta^n
\eeqn
holds for all $n\ge 0$.
\end{prop}

\begin{proof}
Observe that
\beqn
\bP(N_n < [t\alpha n]) = \bP\!\left(\sum_{j=1}^{[t\alpha n]} \tau_j > n\right),
\eeqn
where the equality holds because the sum appearing on the right side is just the time when the~$[t\alpha n]$th coupling occurs. The variables~$\tau_j$,~$j\ge 1$, are independent and~$\tau_j$,~$j\ge 2$, are also identically distributed. Thus, for any $p\in (0,1)$,
\beqn
\begin{split}
\bP\!\left(\sum_{j=1}^{[t\alpha n]} \tau_j > n\right) 
\le p^n\, \bE\!\left(p^{-\sum_{j=1}^{[t\alpha n]} \tau_j}\right)
= p^n\, \bE\!\left(p^{- \tau_1}\right) \! \left(\bE\!\left(p^{- \tau_2}\right)\right)^{[t\alpha n]-1} \le p^n\, \bE\!\left(p^{- \tau_1}\right)\! \left(\bE\!\left(p^{- \tau_2}\right)\right)^{t\alpha n}
\end{split} 
\eeqn
By Proposition~\ref{prop:coupling_time_tail}, each of the random variables $\tau_j$ has an exponential tail: more precisely, there exists $q\in(0,1)$ such that
\beqn
\bP(\tau_1>m) \le \frac{(K'')^{1/\alpha}}{K-1} q^m \quad\text{and}\quad \bP(\tau_2>m) \le \frac{(K')^{1/\alpha}}{K-1} q^m
\eeqn
for $m\ge 0$. Thus, fixing any $p\in(q,1)$,
\beqn
\bE\!\left(p^{- \tau_j}\right) = \int_1^\infty \bP(p^{-\tau_j}>x)\,\rd x  = \int_1^\infty \bP\!\left(\tau_j >\frac{\log x}{\log(p^{-1})}\right)\,\rd x   < \infty.
\eeqn
In fact, writing
\beqn
c = \int_1^\infty x^{-\log q/\log p}\,\rd x,
\eeqn
we have
\beqn
\bE\!\left(p^{- \tau_1}\right)  \le \frac{(K'')^{1/\alpha}c}{K-1}  \quad\text{and}\quad \bE\!\left(p^{- \tau_2}\right)  \le \frac{(K')^{1/\alpha}c}{K-1}.
\eeqn
We then have
\beqn
\bP(N_n < [u n])  \le  \frac{(K'')^{1/\alpha}c}{K-1}\left(p\biggl(\frac{(K')^{1/\alpha}c}{K-1}\biggr)^u\,\right)^n.
\eeqn
Moreover, with the choices
\beqn
p = q^{2\beta},\qquad \beta = \frac{K-1}{2K},\qquad u = t\alpha
\eeqn
we have $c = K-1$ and
\beqn
\bP(N_n < [t\alpha n])  \le  (K'')^{1/\alpha}\bigl(q^{2\beta}(K')^{t}\bigr)^{n}.
\eeqn
Now we choose $t > 0$ so that $(K')^t = q^{-\beta}$, which yields $\vartheta = q^\beta < 1$ and $D = (K'')^{1/\alpha}$. 
\end{proof}

Proposition~\ref{prop:N_n} implies the bound
\beqn
\begin{split}
\bP(\text{$N_n < [t\alpha n]$ for some $n\ge m$}) 
& = \bP\!\left(\bigcup_{n\ge m} \{N_n < [t\alpha n]\}\right)
\le  \sum_{n\ge m}\bP(N_n < [t\alpha n])
\\
& \le \sum_{n\ge m}D \vartheta^{n}
 =  D'\vartheta^{m}
\end{split}
\eeqn
for all $m\ge 0$.
Next, define the random time $\tilde n = \tilde n(K'',\omega)$ by
\beqn
\tilde n = \inf\{m\ge 0\,:\,\text{$N_n \ge [t\alpha n]$ for all $n\ge m$}\}.
\eeqn
In words, given a sequence, the number of couplings by time $n$ is at least $[t\alpha n]$ for \emph{every}~$n\ge \tilde n$.
Then
\beqn
\bP(\tilde n > k) = \bP(\text{$N_n < [t\alpha n]$ for some $n\ge k$}) \le D'\vartheta^k.
\eeqn
In particular, the expected value of $\tilde n$ is finite.

\medskip
\subsection{Proofs of the theorems}
In this section we patch together the results of the previous sections. This leads to Theorems~\ref{thm:seq_bounds} and~\ref{thm:bounds}.

\medskip
Given two probability densities $\psi^1,\psi^2\in\cH_{K''}$, we have
\beqn
\|\cL_{\omega_n}\cdots \cL_{\omega_1} (\psi^1-\psi^2)\|_{L^1(\fm)} \le 2(1-\kappa)^{[t\alpha n]} \le 2(1-\kappa)^{t\alpha n-1} 
\eeqn
for all $n\ge\tilde n$.
Thus, setting
\beqn
\chi(n) = \chi(n;K'',\omega) = 1_{\{\tilde n > n\}} + (1-\kappa)^{t\alpha n-1},
\eeqn
the bound
\beqn
\|\cL_{\omega_n}\cdots \cL_{\omega_1} (\psi^1-\psi^2)\|_{L^1(\fm)} \le 2\chi(n)
\eeqn
holds true for any $n\ge 0$. This implies~\eqref{eq:seq_conv} for the restricted class of densities.

Next, we relax the regularity condition. To this end, given an arbitrary probability density~$\psi\in C^\alpha$, define
\beqn
\psi_h = \frac{\psi+h}{1+h}, \quad h>0.
\eeqn
Since $\psi+h \ge h$, we have
\beqn
|{\log \psi_h(x) - \log\psi_h(y)} | = |{\log(\psi(x)+h) - \log(\psi(y)+h)} | \le \frac{1}{h} |\psi(x)-\psi(y)|\le \frac{|\psi|_\alpha}{h}d(x,y)^\alpha.
\eeqn
Thus, we obtain
\beqn
\psi_h \in \cH_1, \quad h\ge|\psi|_\alpha.
\eeqn
Recall $\abs{\log \phi}_\alpha \le \mathrm{Lip}(\phi) < \infty$.
Setting $h = |\psi|_\alpha + \mathrm{Lip}(\phi)$, both $\psi_h,\phi_h\in\cH_1$, so that
\beqn
\|\cL_{\omega_n}\cdots \cL_{\omega_1} (\psi-\phi)\|_{L^1(\fm)} \le 2(1+h)\chi(n;1,\omega).
\eeqn
As $\|\psi\|_\infty\ge 1$, we can estimate $1+h \le (1+\mathrm{Lip}(\phi))\|\psi\|_\alpha$. Taking expectations,
\beqn
\|\cP^n (\psi-\phi)\|_{L^1(\fm)} \le 2(1+\mathrm{Lip}(\phi))\|\psi\|_\alpha(D'\vartheta^{n} + (1-\kappa)^{t\alpha n-1}).
\eeqn
In other words, we have proved~\eqref{eq:seq_conv} and~\eqref{eq:conv}.

Let us continue our analysis of individual sequences~$\omega$. Suppose $g\in L^\infty$ is complex-valued and $f\in C^\alpha$ is real-valued with
$
\int f\,\rd \fm = 0.
$
Define
\beqn
\tilde f = \frac{f+2|f|_\alpha}{2|f|_\alpha}.
\eeqn
Since $\|f\|_\infty\le|f|_\alpha$, it is easy to check that $\tilde f\in\cH_1$.
Therefore,
\beqn
\begin{split}
\left|\int f\cdot g\circ T_{\omega_n}\circ\cdots\circ T_{\omega_1}\,\rd\fm \right| 
& = \left| \int \cL_{\omega_n}\cdots \cL_{\omega_1} f  \cdot g \,\rd\fm \right| 
 \le \|g\|_\infty \| \cL_{\omega_n}\cdots \cL_{\omega_1} f \|_{L^1(\fm)}
\\
& = 2|f|_\alpha\|g\|_\infty \| \cL_{\omega_n}\cdots \cL_{\omega_1} (\tilde f - \bfone) \|_{L^1(\fm)}
= 4|f|_\alpha\|g\|_\infty \chi(n;1,\omega).
\end{split}
\eeqn
In general, $\int f\,\rd\fm = 0$ fails, in which case the preceding bound yields
\beqn
\left|\int f\cdot g\circ T_{\omega_n}\circ\cdots\circ T_{\omega_1}\,\rd\fm - \int f\,\rd\fm \int g\circ T_{\omega_n}\circ\cdots\circ T_{\omega_1}\,\rd\fm \right| \le 4|f|_\alpha\|g\|_\infty \chi(n;1,\omega).
\eeqn
We can also change the measure in the integrals above. Indeed, let $\psi\in C^\alpha$ be a probability density and denote $\rd\nu = \psi\,\rd\fm$. Then readily
\beqn
\left|\int f\cdot g\circ T_{\omega_n}\circ\cdots\circ T_{\omega_1}\,\rd\nu - \int f\,\rd\nu \int g\circ T_{\omega_n}\circ\cdots\circ T_{\omega_1}\,\rd\fm \right| \le  4\|\psi\|_\alpha \|f\|_\alpha\|g\|_\infty \chi(n;1,\omega),
\eeqn
because $|f\psi|_\alpha\le\|\psi\|_\alpha \|f\|_\alpha$. On the other hand,
\beqn
\begin{split}
& \left| \int g\circ T_{\omega_n}\circ\cdots\circ T_{\omega_1}\,\rd\nu - \int g\circ T_{\omega_n}\circ\cdots\circ T_{\omega_1}\,\rd\fm \right|
\\
=\ & \left| \int \psi\cdot g\circ T_{\omega_n}\circ\cdots\circ T_{\omega_1}\,\rd\fm - \int\psi\,\rd\fm \int g\circ T_{\omega_n}\circ\cdots\circ T_{\omega_1}\,\rd\fm \right|
\\
\le\ & 4 |\psi|_\alpha\|g\|_\infty \chi(n;1,\omega).
\end{split}
\eeqn
Collecting the bounds,
\beqn
\left|\int f\cdot g\circ T_{\omega_n}\circ\cdots\circ T_{\omega_1}\,\rd\nu - \int f\,\rd\nu \int g\circ T_{\omega_n}\circ\cdots\circ T_{\omega_1}\,\rd\nu \right| \le 8\|\psi\|_\alpha \|f\|_\alpha\|g\|_\infty \chi(n;1,\omega).
\eeqn
Hence, we have proved~\eqref{eq:seq_mix} for real-valued $f$. For complex-valued $f$, a similar bound follows from the one above with a larger prefactor. This proves Theorem~\ref{thm:seq_bounds}.

In particular, we can choose $\nu = \mu$. Since $f$ and $g$ are bounded, we can therefore estimate
\beqn
\begin{split}
& \left|\int f\cdot \cQ^n g\,\rd\mu -\int f\,\rd\mu\int g\,\rd\mu \right| = \left|\int f\cdot \cQ^n g\,\rd\mu -\int f\,\rd\mu\int \cQ^n g\,\rd\mu \right| 
\\
=\ &
\left|\int f\cdot \bE\!\left[g\circ T_{\omega_n}\circ\cdots\circ T_{\omega_1}\right]\,\rd\mu -\int f\,\rd\mu\int \bE\!\left[g\circ T_{\omega_n}\circ\cdots\circ T_{\omega_1}\right]\,\rd\mu \right| 
\\
\le\ &
\bE\!\left[
\left|\int f\cdot g\circ T_{\omega_n}\circ\cdots\circ T_{\omega_1}\,\rd\mu -\int f\,\rd\mu\int g\circ T_{\omega_n}\circ\cdots\circ T_{\omega_1}\,\rd\mu \right|\right]
\\
\le\ &
C\|\phi\|_\alpha \|f\|_\alpha\|g\|_\infty E[\chi(n;1,\omega)]
\le C\|\phi\|_\alpha \|f\|_\alpha\|g\|_\infty( D'\vartheta^{n} + (1-\kappa)^{t\alpha n-1}).
\end{split}
\eeqn
This proves~\eqref{eq:mix} and Theorem~\ref{thm:bounds}.
\qed


\medskip
\section{Proof of Theorem~\ref{thm:VASIP}}

Our proof of Theorem~\ref{thm:VASIP} is based on providing exponential bounds, uniform in $n$, on multiple correlation functions of the form
\beq\label{eq:corr}
E^\mu[F G_n] - E^\mu[F]E^\mu[G_n],
\eeq
where, for certain Hölder continuous functions $g_i\in C^\alpha$, $i\ge 0$,
\beq\label{eq:F&G}
\begin{split}
F & = g_0\circ X_0\cdots g_m\circ X_m
\\
G_n & = g_{m+1}\circ X_{m+1+n}\cdots g_{m+k}\circ X_{m+k+n} \ .
\end{split}
\eeq
The main ingredient for obtaining such bounds will be the pair correlation bound in~\eqref{eq:mix} of Theorem~\ref{thm:bounds}. Here, beside the uniform exponential rate, the crucial bit of information is that the function~$g$ appearing in~\eqref{eq:mix} is only required to be in $L^\infty$ and that the bound depends on $g$ only through its $L^\infty$ norm.

\medskip

Fix $\alpha\in(0,1)$, $H>0$ and $\ve>0$.
Let $f_k$, $k\ge 0$, be real-valued functions such that
\beq\label{eq:f_cond}
\sup_{k\ge 0}|f_k|_\alpha \le H.
\eeq
Let $t_k$, $k\ge 0$, be real numbers satisfying
\beq\label{eq:t_cond}
\sup_{k\ge 0}|t_k|\leq \ve
\eeq
and define the functions
\beqn
g_k = e^{i t_k f_k}, \quad k\in \bN.
\eeqn
These are the functions we use in~\eqref{eq:F&G}. Notice immediately that
\beq\label{eq:g_bounds}
|g_k| = 1 \quad\text{and}\quad |g_k|_\alpha\le \ve H.
\eeq

For what follows, we define the operator $\hat\cP$ by setting
\beqn
\hat \cP g = \phi^{-1}\cP(\phi g).
\eeqn
This will be convenient for manipulating integrals with resect to the invariant measure $\mu$, as
\beqn
\int \hat\cP g\cdot f\,\rd \mu = \int \cP (\phi g)\cdot f\,\rd \fm = \int g\cdot \cQ f\,\rd \mu .
\eeqn
We also introduce the operators $\hat\cP_g$ and $\cQ_g$ which act according to
\beqn
\hat\cP_g(h) = \hat\cP(gh)
\quad\text{and}\quad
\cQ_g(h) = \cQ(g h).
\eeqn

\begin{lem}
Defining
\beqn
\widetilde G_n(\omega,x) = g_{m+1}\circ X_{1+n}(\sigma^m\omega,x)\cdots g_{m+k}\circ X_{k+n}(\sigma^m \omega,x),
\eeqn
we have
\beq\label{eq:corr_temp1}
E^\mu[F G_n] =  E^\mu \Bigl[g_m\hat\cP_{g_{m-1}}\cdots\hat\cP_{g_1}\hat\cP_{g_0} \bfone \cdot \widetilde G_n \Bigr].
\eeq
Above, the operator product $\hat\cP_{g_{m-1}}\cdots\hat\cP_{g_1}\hat\cP_{g_0}$ acts on the constant function $\bfone$. Moreover,
\beq\label{eq:corr_temp2}
E^\mu[G_n] = E^\mu\Bigl[\widetilde G_n \Bigr] 
\quad\text{and}\quad
E^\mu[F] = E^\mu\Bigl[g_m\hat\cP_{g_{m-1}}\cdots\hat\cP_{g_1}\hat\cP_{g_0} \bfone \Bigr].
\eeq
\end{lem}
\begin{proof}
We can write
\beqn
X_{m+l}(\omega,x) = X_{l}(\sigma^m\omega,X_m(\omega,x)).
\eeqn
Accordingly, $\widetilde G_n(\omega,X_m(\omega,x)) = G_n(\omega,x)$. Because the Markov chain $(X_n)_{n\geq 0}$ is stationary, $X_m$ has distribution $\mu$. Now, since $\widetilde G_n(\omega,\slot)$ only depends on $\omega_i$, $i>m$, the first identity in~\eqref{eq:corr_temp2} follows. Next, integrating with respect to the variables~$\omega_i$, $1\leq i\leq m$, in ascending order of the index~$i$, we get
\beqn
E^\mu[F G_n] =  E^\mu \Bigl[g_0\cQ_{g_1}\cQ_{g_2} \cdots \cQ_{g_m} \widetilde G_n \Bigr].
\eeqn
Using duality repeatedly, starting with the first $\cQ$ from the left, then the second, and so on, we arrive ultimately at \eqref{eq:corr_temp1}. The second identity in~\eqref{eq:corr_temp2} is proved in a similar fashion.
\end{proof}

As a consequence, the difference in~\eqref{eq:corr} equals
\beqn
E^\mu \Bigl[g_m\hat\cP_{g_{m-1}}\cdots\hat\cP_{g_1}\hat\cP_{g_0} \bfone \cdot \widetilde G_n \Bigr] - E^\mu \Bigl[g_m\hat\cP_{g_{m-1}}\cdots\hat\cP_{g_1}\hat\cP_{g_0} \bfone \Bigr]E^\mu \Bigl[\widetilde G_n \Bigr] .
\eeqn
Note also that $g_m\hat\cP_{g_{m-1}}\cdots\hat\cP_{g_1}\hat\cP_{g_0} \bfone$ does not depend on $\omega$ at all. Thus, 
\beqn
E^\mu \Bigl[g_m\hat\cP_{g_{m-1}}\cdots\hat\cP_{g_1}\hat\cP_{g_0} \bfone \Bigr] 
= \int g_m\hat\cP_{g_{m-1}}\cdots\hat\cP_{g_1}\hat\cP_{g_0} \bfone \,\rd\mu.
\eeqn
We can also integrate out the $\omega$-dependence of $\widetilde G_n$:
\beqn
\int \widetilde G_n \, \rd\eta^{k+n}(\omega_{m+1},\cdots,\omega_{m+k+n}) = \cQ^n \cQ_{g_{m+1}} \cQ_{g_{m+2}} \cdots \cQ_{g_{m+k}}\bfone.
\eeqn
Here the $\omega_i$-integrals were done in descending order of the index $i$. The resulting expression only depends on $x$. This leaves us with
\beq\label{eq:multi_bound}
\begin{split}
& E^\mu[F G_n] - E^\mu[F]E^\mu[G_n] 
=
\int g_m\hat\cP_{g_{m-1}}\cdots\hat\cP_{g_1}\hat\cP_{g_0} \bfone \cdot \cQ^n \cQ_{g_{m+1}} \cQ_{g_{m+2}} \cdots \cQ_{g_{m+k}}\bfone\,\rd\mu 
\\
&\qquad\qquad\qquad\qquad\qquad
 -\int g_m\hat\cP_{g_{m-1}}\cdots\hat\cP_{g_1}\hat\cP_{g_0} \bfone \,\rd\mu \cdot
\int  \cQ^n \cQ_{g_{m+1}} \cQ_{g_{m+2}} \cdots \cQ_{g_{m+k}}\bfone\,\rd\mu .
\end{split}
\eeq

\medskip

In order to take advantage of~\eqref{eq:mix} directly, we will need to bound $\cQ_{g_{m+1}} \cQ_{g_{m+2}} \cdots \cQ_{g_{m+k}}\bfone$ in the supremum norm and $g_m\hat\cP_{g_{m-1}}\cdots\hat\cP_{g_1}\hat\cP_{g_0} \bfone$ in the Hölder norm. Bounds in the supremum norm are immediate, because $\cQ$ and $\cP$ are increasing operators and because $\|g_i\|_\infty = 1$. Indeed,
\beqn
|\cQ_{g_i} h| \le \cQ|g_i h| \le \|h\|_\infty \cQ\bfone = \|h\|_\infty
\eeqn
and
\beqn
|\hat\cP_{g_i} h| \le \phi^{-1}\cP(\phi |g_i h|) \leq \|h\|_\infty\phi^{-1}\cP(\phi) = \|h\|_\infty
\eeqn
for any $h:\bS\to\bC$, so that
\beq\label{eq:Q_prod_bound}
\|\cQ_{g_{m+1}} \cQ_{g_{m+2}} \cdots \cQ_{g_{m+k}}\bfone\|_\infty \le 1
\eeq
and
\beq\label{eq:P_prod_bound}
\|g_m\hat\cP_{g_{m-1}}\cdots\hat\cP_{g_1}\hat\cP_{g_0} \bfone\|_\infty \le \|\hat\cP_{g_{m-1}}\cdots\hat\cP_{g_1}\hat\cP_{g_0} \bfone\|_\infty \le 1.
\eeq
We now proceed to bounding the Hölder constant $|g_m\hat\cP_{g_{m-1}}\cdots\hat\cP_{g_1}\hat\cP_{g_0} \bfone|_\alpha$, which is more subtle.

\medskip

Note that
\beq\label{eq:P_prod}
\hat\cP_{g_{n-1}}\cdots\hat\cP_{g_0} \bfone = \phi^{-1}\cP_{g_{n-1}}\cdots \cP_{g_0} \phi = \phi^{-1} \int \cL_{\omega_n,g_{n-1}}\cdots \cL_{\omega_1,g_0} \phi \,\rd\eta^n(\omega_1,\dots,\omega_n),
\eeq
where 
\beqn
\cL_{\omega_i,g} h = \cL_{\omega_i} (g h).
\eeqn
The following identity will be convenient, because the right side involves a composition of the ``usual'' transfer operators $\cL_{\omega_i}$:

\begin{lem}\label{lem:L_prod}
For any $h$,
\beqn
\cL_{\omega_n,g_{n-1}}\cdots \cL_{\omega_1,g_0} h = \cL_{\omega_n}\cdots\cL_{\omega_1}(e^{V_n} h),
\eeqn
where $V_n = V_n(\omega) = \sum_{k=0}^{n-1} it_kf_k\circ T_{\omega_k}\circ\dots\circ T_{\omega_1}$. 
\end{lem}

\begin{proof}
This holds for $n=1$. Assume that it holds for $n=k$. Then, for $n=k+1$ and any~$u$,
\beqn
\begin{split}
& \int u\cdot \cL_{\omega_n,g_{n-1}}\cdots \cL_{\omega_1,g_0} h \,\rd\fm 
 =  \int u\cdot \cL_{\omega_n,g_{n-1}} \cL_{\omega_{n-1}}\cdots\cL_{\omega_1}(e^{V_{n-1}} h)\,\rd\fm
\\
& \qquad\qquad\qquad =  \int u\circ T_{\omega_n} e^{it_{n-1}f_{n-1}} \cL_{\omega_{n-1}}\cdots\cL_{\omega_1}(e^{V_{n-1}} h)\,\rd\fm
\\
& \qquad\qquad\qquad =  \int u\circ T_{\omega_n}\circ\cdots \circ T_{\omega_1} \exp(it_{n-1}f_{n-1}\circ T_{\omega_{n-1}}\circ\cdots \circ T_{\omega_1}) \cdot e^{V_{n-1}} h\,\rd\fm
\\
& \qquad\qquad\qquad =  \int u\circ T_{\omega_n}\circ\cdots \circ T_{\omega_1} \cdot e^{V_{n}} h\,\rd\fm =   \int u\cdot \cL_{\omega_n}\cdots\cL_{\omega_1}(e^{V_{n}} h)\,\rd\fm.
\end{split}
\eeqn
Thus, the induction principle proves the lemma.
\end{proof}

\begin{lem}\label{lem:L_prod_bound}
For any complex-valued $h\in C^\alpha$,
\beqn
|\cL_{\omega_n,g_{n-1}}\cdots \cL_{\omega_1,g_0} h|_\alpha
\le (1+R_n)\left( \sum_{j = 1}^n \frac{\|h\|_\infty \Delta_{\omega_j}}{\lambda_{\omega_n}\cdots\lambda_{\omega_{j+1}}}  +  \sum_{k=0}^{n-1} \frac{\ve H\|h\|_\infty}{\lambda_{\omega_n}^{\alpha}\cdots \lambda_{\omega_{k+1}}^{\alpha}} + \frac{|h|_\alpha}{\lambda_{\omega_n}^{\alpha}\cdots \lambda_{\omega_1}^{\alpha}} \right) .
\eeqn
\end{lem}

\begin{proof}
Consider two points $x,y\in\bS$ and an arc $J$ containing both $x$ and $y$ with $|J|\le \frac12$. Denote by $\cS_{n,i}:J\to \cS_{n,i}J$, $1\le i\le w$, the branches of the inverse of $\cT_n \equiv T_{\omega_n}\circ\dots\circ T_{\omega_1}$. Note
\beqn
 \frac{d}{dz} \frac{1}{\cT_n'(\cS_{n,i} z)} = - \frac{\cT_n''(\cS_{n,i} z)\cdot \cS_{n,i}'(z)}{(\cT_n'(\cS_{n,i} z))^2} = - \frac{\cT_n''(\cS_{n,i} z)}{(\cT_n'(\cS_{n,i} z))^3} .
\eeqn
Observe also that $\cL_{\omega_n}\cdots\cL_{\omega_1}$ is the transfer operator associated to $\cT_n$. Without loss of generality, we assume $\cT_n'>0$. Then, recalling Lemma~\ref{lem:L_prod} and that $|e^{V_n}| = 1$,
\begin{align*}
& \left|\cL_{\omega_n}\cdots\cL_{\omega_1}(e^{V_n} h)(x)-\cL_{\omega_n}\cdots\cL_{\omega_1}(e^{V_n} h)(y)\right|
= \left|\sum^{w}_{i=1}\frac{(e^{V_n} h)(\cS_{n,i} x)}{\cT_n'(\cS_{n,i} x)}-\sum^{w}_{i=1}\frac{(e^{V_n} h)(\cS_{n,i} y)}{\cT_n'(\cS_{n,i} y)} \right|
\\
&
\qquad\leq \left|\sum^{w}_{i=1}\left(\frac{1}{\cT_n'(\cS_{n,i} x)}-\frac{1}{\cT_n'(\cS_{n,i} y)}\right) (e^{V_n} h)(\cS_{n,i} x)\right|
\\
&\qquad\qquad\qquad\qquad\qquad\qquad
	+\left|\sum^{w}_{i=1}\frac{1}{\cT_n'(\cS_{n,i} y)} \left((e^{V_n} h)(\cS_{n,i} x) - (e^{V_n} h)(\cS_{n,i} y)\right)  \right|
\\
&
\qquad\leq \|e^{V_n} h\|_\infty \sum^{w}_{i=1}\left|\frac{1}{\cT_n'(\cS_{n,i} x)}-\frac{1}{\cT_n'(\cS_{n,i} y)}\right|
	+ \sum^{w}_{i=1}\frac{1}{\cT_n'(\cS_{n,i} y)} |(e^{V_n} h)(\cS_{n,i} y)-(e^{V_n} h)(\cS_{n,i} x)|
\\
&
\qquad\leq \|h\|_\infty \sum^{w}_{i=1}\left|\int \frac{\cT_n''(\cS_{n,i} z)}{(\cT_n'(\cS_{n,i} z))^3} \,\rd\fm(z) \, d(x, y)\right|
	+ \sum^{w}_{i=1}\frac{|e^{V_n\circ \cS_{n,i}}h\circ \cS_{n,i}|_\alpha}{\cT_n'(\cS_{n,i} y)}d(x,y)^\alpha
\\
&
\qquad\leq \|h\|_\infty  \left \|\frac{\cT_n''}{(\cT_n')^2}\right\|_\infty \int \cL_{\omega_n}\cdots\cL_{\omega_1}\bfone \,\rd\fm \, d(x, y)
\\
&\qquad\qquad\qquad\qquad\qquad\qquad
	+  \sup_{1\leq i \leq w}|e^{V_n\circ \cS_{n,i}}h\circ \cS_{n,i}|_\alpha  d(x,y)^\alpha \cL_{\omega_n}\cdots\cL_{\omega_1}\bfone(y)
\\
&
\qquad\leq \left(\|h\|_\infty \left\| \frac{\cT_n''}{(\cT_n')^2} \right\|_\infty
	+ \sup_{1\leq i \leq w}|e^{V_n\circ \cS_{n,i}}h\circ \cS_{n,i}|_\alpha\right)d(x,y)^\alpha \cdot \|\cL_{\omega_n}\cdots\cL_{\omega_1}\bfone\|_\infty .
\end{align*}
Now, we can first estimate $\sup_{1\leq i \leq w}|e^{V_n\circ \cS_{n,i}}h\circ \cS_{n,i}|_\alpha$:
\beqn
\begin{split}
|e^{V_n\circ \cS_{n,i}}h\circ \cS_{n,i}|_\alpha 
&\le \|h\|_\infty |e^{V_n\circ \cS_{n,i}}|_\alpha + |h\circ \cS_{n,i}|_\alpha
\le \|h\|_\infty \sum_{k=0}^{n-1}|t_k| |f_k\circ\cT_k\circ \cS_{n,i}|_\alpha + |h|_\alpha \lambda_{\cT_n}^{-\alpha}
\\
&\le \ve H\|h\|_\infty  \sum_{k=0}^{n-1}  \lambda_{\omega_n}^{-\alpha}\cdots \lambda_{\omega_{k+1}}^{-\alpha} + |h|_\alpha \lambda_{\omega_n}^{-\alpha}\cdots \lambda_{\omega_1}^{-\alpha} .
\end{split}
\eeqn
Finally,
\beqn
\begin{split}
 \left\| \frac{\cT_n''}{(\cT_n')^2} \right\|_\infty
& = \left \| \frac{1}{\cT_n'} \left(\log \cT_n'\right)' \right\|_\infty
= \left\| \frac{1}{\cT_n'} \sum_{j = 1}^n \left(\log T_{\omega_j}'\circ \cT_{j-1}\right)' \right\|_\infty 
\\
& \le \sum_{j = 1}^n \left\| \frac{1}{(T_{\omega_n}\circ\cdots\circ T_{\omega_{j+1}})'\circ\cT_j\cdot \cT_j'} \cdot \frac{T_{\omega_j}''\circ \cT_{j-1}\cdot \cT_{j-1}'}{T_{\omega_j}'\circ \cT_{j-1}} \right\|_\infty
\\
& = \sum_{j = 1}^n  \left\| \frac{1}{(T_{\omega_n}\circ\cdots\circ T_{\omega_{j+1}})'\circ\cT_j} \cdot \frac{T_{\omega_j}''}{(T_{\omega_j}')^2}\circ \cT_{j-1} \right\|_\infty 
\le \sum_{j = 1}^n \frac{1}{\lambda_{\omega_n}\cdots\lambda_{\omega_{j+1}}} \Delta_{\omega_j} \ .
\end{split}
\eeqn
Collecting the bounds and recalling~\eqref{eq:L_prod_sup} finishes the proof.
\end{proof}

\begin{prop}\label{prop:Holder_bound}
Given $\alpha\in(0,1)$, $H>0$ and $\ve>0$, there exists such a constant $C>0$ that
\beqn
\sup_{n\ge 0}|g_n\hat\cP_{g_{n-1}}\cdots\hat\cP_{g_0} \bfone|_\alpha
\leq C
\eeqn
holds for all $n \ge 0$, for all choices of $(f_k)_{k\ge 0}$ and $(t_k)_{k\ge 0}$ satisfying~\eqref{eq:f_cond} and~\eqref{eq:t_cond}.
\end{prop}

\begin{proof}
First,
\begin{align*}
\begin{split}
|g_n\hat\cP_{g_{n-1}}\cdots\hat\cP_{g_0} \bfone|_\alpha
&\leq |g_n|_\alpha \|\hat\cP_{g_{n-1}}\cdots\hat\cP_{g_0} \bfone\|_\infty + |\hat\cP_{g_{n-1}}\cdots\hat\cP_{g_0} \bfone|_\alpha
\\
& \le \ve H + |\hat\cP_{g_{n-1}}\cdots\hat\cP_{g_0} \bfone|_\alpha
\end{split}
\end{align*}
where \eqref{eq:g_bounds} and~\eqref{eq:P_prod_bound} were used. Next, by~\eqref{eq:P_prod} and Lemma~\ref{lem:L_prod},
\begin{align*}
|\hat\cP_{g_{n-1}}\cdots\hat\cP_{g_0} \bfone|_\alpha &\leq \left| \phi^{-1} \int \cL_{\omega_n,g_{n-1}}\cdots \cL_{\omega_1,g_0} \phi \,\rd\eta^n(\omega_1,\dots,\omega_n)\right|_\alpha
\\
&\leq |\phi^{-1}|_\alpha\|\phi\hat\cP_{g_{n-1}}\cdots\hat\cP_{g_0} \bfone\|_\infty + \|\phi^{-1}\|_\infty\int |\cL_{\omega_n,g_{n-1}}\cdots \cL_{\omega_1,g_0} \phi|_\alpha \,\rd\eta^n(\omega_1,\dots,\omega_n)
\\
&\leq |\phi|_\alpha \|\phi^{-1}\|_{\infty}^2 \|\phi\|_\infty+ \|\phi^{-1}\|_\infty\int |\cL_{\omega_n,g_{n-1}}\cdots \cL_{\omega_1,g_0} \phi|_\alpha \,\rd\eta^n(\omega_1,\dots,\omega_n).
\end{align*}

With the aid of Lemma \ref{lem:L_prod_bound}, the bound in~\eqref{eq:L_prod_sup} and Lemma~\ref{lem:R^2_bound}, we can bound the integral in the last line above using Hölder's inequality and independence. Namely,
\begin{align*}
&\quad \int |\cL_{\omega_n,g_{n-1}}\cdots \cL_{\omega_1,g_0} \phi|_\alpha \,\rd\eta^n(\omega_1,\dots,\omega_n)
\\
&\leq \bE\!\left[ (1+R_n)\left(\sum_{j = 1}^n \frac{\|\phi\|_\infty \Delta_{\omega_j}}{\lambda_{\omega_n}\cdots\lambda_{\omega_{j+1}}}  +  \sum_{k=0}^{n-1} \frac{\ve H\|\phi\|_\infty }{\lambda_{\omega_n}^{\alpha}\cdots \lambda_{\omega_{k+1}}^{\alpha}} + \frac{|\phi|_\alpha}{\lambda_{\omega_n}^{\alpha}\cdots \lambda_{\omega_1}^{\alpha}}\right)\right]
\\
&\leq \left(\bE\!\left[(1+R_n)^2\right]\right)^{1/2}\left\{\|\phi\|_\infty \sum_{j = 1}^n \left(\bE\!\left[\frac{\Delta_{\omega_j}^2}{\lambda_{\omega_n}^2\cdots\lambda_{\omega_{j+1}}^2}\right]\right)^{1/2}\right.
\\
&\qquad\qquad\qquad\qquad\qquad +  \left.\ve H \|\phi\|_\infty  \sum_{k=0}^{n-1} \left(\bE\!\left[\frac{1}{\lambda_{\omega_n}^{2\alpha}\cdots \lambda_{\omega_{k+1}}^{2\alpha}}\right]\right)^{1/2} + |\phi|_\alpha \left(\bE\!\left[\frac{1}{\lambda_{\omega_n}^{2\alpha}\cdots \lambda_{\omega_1}^{2\alpha}}\right]\right)^{1/2}\right\}
\\
&\leq C_R^{1/2}\!\left(\|\phi\|_\infty  \langle\Delta^2\rangle^{1/2}\sum_{j = 1}^n  \langle\lambda^{-2}\rangle^{(n-j)/2} + \ve H\|\phi\|_\infty \sum_{k=0}^{n-1} \langle\lambda^{-2\alpha}\rangle^{(n-k)/2} 
 + |\phi|_\alpha \langle\lambda^{-2\alpha}\rangle^{n/2} \right)
\\
&\leq C_R^{1/2} \|\phi\|_\alpha\!   \left( \langle\Delta^2\rangle^{1/2}\sum_{j = 1}^n  \langle\lambda^{-2}\rangle^{(n-j)/2} 
+ \ve H \sum_{k=0}^{n-1} \langle\lambda^{-2\alpha}\rangle^{(n-k)/2} 
 +  \langle\lambda^{-2\alpha}\rangle^{n/2} \right).
\end{align*}
By Jensen's inequality, $\langle\lambda^{-2\alpha}\rangle \le \langle\lambda^{-2}\rangle^\alpha$. Since also $\langle\lambda^{-2}\rangle\le \langle\lambda^{-2}\rangle^\alpha<1$ holds, we can further bound the expression in the last line by
\beqn
C_R^{1/2} \|\phi\|_\alpha\! \left(\frac{\langle\Delta^2\rangle^{1/2} 
+ \ve H}{1-\langle\lambda^{-2}\rangle^{\alpha/2}} 
+ 1\right).
\eeqn
Finally, note that $|\phi|_\alpha$ is bounded by the Lipschitz constant of $\phi$ for all $\alpha\in(0,1)$.
\end{proof}

\medskip
We are finally in position to state the multiple correlation bound that is needed to prove Theorem~\ref{thm:VASIP}:
\begin{thm}\label{thm:multi}
There exist such a constant $\theta\in (0,1)$ and, given $\alpha\in(0,1)$, $H>0$ and $\ve>0$, a constant $C>0$ that 
\beqn
\left| E^\mu[F G_n] - E^\mu[F]E^\mu[G_n] \right| \le C\theta^{\alpha n}
\eeqn
holds for all $n\ge 0$, for all choices of $(f_k)_{k\ge 0}$ and $(t_k)_{k\ge 0}$ satisfying~\eqref{eq:f_cond} and~\eqref{eq:t_cond}.
\end{thm}
\begin{proof}
This follows immediately from~\eqref{eq:mix} of Lemma~\ref{thm:bounds}, once we collect~\eqref{eq:multi_bound}, \eqref{eq:Q_prod_bound}, \eqref{eq:P_prod_bound} and Proposition~\ref{prop:Holder_bound}.
\end{proof}

\medskip
The following theorem is a special case of the main result in~\cite{Gouezel_2010}.
\begin{thm}[Gou\"ezel~{\cite{Gouezel_2010}}]
Let $(\bfA_n)_{n\geq 0}$ be a stationary sequence of $\bR^d$-valued random variables which is centered and bounded. Given integers $n>0$, $m>0$, $0\leq b_1<b_2<\dots<b_{n+m+1}$, $k\geq 0$, and vectors $\bft_1,\dots,\bft_{n+m}\in\bR^d$, define
\beqn
X_{n,m}^{(k)} = \sum_{j=n}^m \bft_j\cdot\sum_{l=b_j+k}^{b_{j+1}-1+k}\bfA_l.
\eeqn
Assume that there exist such constants $\ve>0$, $C>0$, and $c>0$ that
\beq\label{eq:Gouezel}
\left|E\!\left(e^{i X_{1,n}^{(0)}+i X_{n+1,n+m}^{(k)}}\right)-E\!\left(e^{i X_{1,n}^{(0)}}\right)\! E\!\left(e^{i X_{n+1,n+m}^{(k)}}\right)\right|
\leq Ce^{-ck}\!\left(1+\max_{1\leq j\leq n+m}|b_{j+1}-b_j|\right)^{C(n+m)}
\eeq
holds for all choices of the numbers $n$, $m$, $b_j$, $k>0$, and of the vectors $\bft_j$ satisfying $|\bft_j|\le\ve$. Then items~(1)--(3) of Theorem~\ref{thm:VASIP} (with $E$ in place of $E^\mu$) are true for the process~$(\bfA_n)_{n\geq 0}$.
\end{thm}
In other words, it is now enough to prove that~\eqref{eq:Gouezel} holds in our case, with $\bfA_n$ as defined in~\eqref{eq:A_n}.
This is immediate, as
\beqn
\bft_j\cdot \bfA_l(\omega,x) = \abs{\bft_j}\left(\frac{\bft_j}{\abs{\bft_j}}\cdot \bff\circ T_{\omega_l}\circ\dots\circ T_{\omega_1}(x)\right),
\eeqn
where the maps $\frac{\bft_j}{\abs{\bft_j}}\cdot \bff:\bS\to\bR$, $j\ge 0$, are uniformly Hölder continuous:
\beqn
\left |\frac{\bft_j}{\abs{\bft_j}}\cdot \bff(x)-\frac{\bft_j}{\abs{\bft_j}}\cdot \bff(y) \right| \le |\bff(x)-\bff(y)| \le |\bff|_\alpha\,d(x,y)^\alpha.
\eeqn
Therefore, the difference on the left side of~\eqref{eq:Gouezel} is of the general form~\eqref{eq:corr} with $F$ and $G_n$ as in~\eqref{eq:F&G}. Theorem~\ref{thm:multi} thus yields the bound $C\theta^{\alpha n}$ on the right side of~\eqref{eq:Gouezel}.

\medskip
The proof of Theorem~\ref{thm:VASIP} is complete.
\qed


\medskip
\section{Proof of Lemma~\ref{lem:degeneracy}}\label{sec:degeneracy}

Let $\Phi$ denote the skew product map $\Phi(\omega,x) = (\sigma\omega,T_{\omega_1}x)$ and $\pi$ the projection $\pi(\omega,x) = x$. Abusing notation, we write $P^\mu = \bP\times\mu$ and $E^\mu$ for the corresponding expectation in this section; this measure is invariant for~$\Phi$.
Setting $\bfA_n = \bff\circ\pi\circ\Phi^n$ (cf.~\eqref{eq:A_n}), observe that
\beq\label{eq:cov_new}
\begin{split}
\bfSigma^2 & = E^\mu(\bfA_0\otimes\bfA_0)+ \sum_{m=1}^{\infty} E^\mu (\bfA_0\otimes\bfA_{m} + \bfA_{m}\otimes\bfA_0).
\end{split}
\eeq
Denote
\beqn
\bfS_n = \sum_{m=0}^{n-1} \bfA_{m} .
\eeqn
\begin{lem}
There exists a constant $C\geq 0$ such that
\beq\label{eq:cov_new_bound}
\sup_{n\geq 0}\left\| E^\mu\! \left(\bfS_n\otimes \bfS_n \right) - n\bfSigma^2 \right\| \leq C.
\eeq
\end{lem}
\begin{proof}
Using invariance,
\beqn
\begin{split}
E^\mu\! \left(\bfS_n\otimes \bfS_n \right) 
& =
n\,E^\mu(\bfA_0\otimes\bfA_0)+ \sum_{m=1}^{n-1} (n-m) \,E^\mu (\bfA_0\otimes\bfA_{m} + \bfA_{m}\otimes\bfA_0).
\end{split}
\eeqn
Recalling~\eqref{eq:cov_new}, we have
\beqn
\begin{split}
E^\mu\! \left(\bfS_n\otimes \bfS_n \right) - n\bfSigma^2 
& = \sum_{m=1}^{\infty} a_n(m) \,E^\mu (\bfA_0\otimes\bfA_{m} + \bfA_{m}\otimes\bfA_0)
\end{split}
\eeqn
where $a_n(m) = -m$ for $1\leq m<n$ and $a_n(m) = -n$ for $m\geq n$. Because $|a_n(m)|\leq m$,
\beqn
\left \|E^\mu\! \left(\bfS_n\otimes \bfS_n \right) - n\bfSigma^2 \right\| \leq \sum_{m=1}^{\infty} m \left\| E^\mu (\bfA_0\otimes\bfA_{m} + \bfA_{m}\otimes\bfA_0) \right\| .
\eeqn
But
\beqn
E^\mu (\bfA_0\otimes\bfA_{m} + \bfA_{m}\otimes\bfA_0) = \int \! \left(\bff\otimes \cQ^m\bff + \cQ^m\bff\otimes\bff\right) \rd\mu
\eeqn
is exponentially small in $m$ according to Theorem~\ref{thm:bounds}, so the sum above converges.
\end{proof}
Recall the $\mu$-average of $\bff$ vanishes.
Given a vector $\bfv\in\bR^d$, we define $f_\bfv = \bfv^\rT \bff$,
\beqn
X_k = f_\bfv\circ\pi\circ\Phi^k
\quad\text{and}\quad
S_n = \sum_{k=0}^{n-1} X_k.
\eeqn
Since $\bfv^\rT(\bfA_m\otimes\bfA_n)\bfv = \bfv^\rT(\bfA_n\otimes\bfA_m)\bfv = X_n X_m$ for all $n,m\ge 0$,~\eqref{eq:cov_new_bound} gives
\beq\label{eq:cov_new_bound2}
\left|E^\mu\!\left(S_n^2 \right) - n\,\bfv^\rT  \bfSigma^2 \bfv\right|  = \left|\bfv^\rT \bigl(E^\mu(\bfS_{n}\otimes \bfS_{n})  - n \bfSigma^2 \bigr)\bfv \right| \leq C|\bfv|^2
\eeq
uniformly for $n\geq 1$. 

\medskip

From here, the proof is similar to~\cite{AyyerLiveraniStenlund_2009}. Suppose $\bfSigma^2$ is degenerate. In other words, there exists a vector $\bfv\in\bR^d$ such that $\bfv^\rT \bfSigma^2 \bfv = \mathbf{0}$.  By the bound above, the random variables
$S_n$
are uniformly bounded in $L^2(P^\mu)$. By the Banach--Alaoglu theorem, there exists a sequence $(n_k)_{k\ge 1}$ and $S\in L^2(P^\mu)$ such that
\beqn
\lim_{k\to\infty} E^\mu(h S_{n_k}) = E^\mu(h S)
\eeqn
for all $h\in L^2(P^\mu)$. In particular, if $h$ is independent of the sequence $\omega$ and $g = \mu(S)$,
\beqn
\lim_{k\to\infty} \sum_{j=0}^{{n_k}-1}  \mu \! \left (h\cQ^j f_\bfv\right ) = \mu (h g).
\eeqn
A similar identity is obtained with $\hat\cP h$ in place of $h$. Therefore,
\beqn
\mu \!\left (h (f_\bfv-g+\cQ g) \right) 
= \mu(hf_\bfv)-\lim_{k\to\infty} \sum_{j=0}^{{n_k}-1}  \mu \! \left (h\cQ^j f_\bfv\right ) + \lim_{k\to\infty} \sum_{j=1}^{{n_k}}  \mu \! \left (h\cQ^j f_\bfv\right )
=  \lim_{k\to\infty}  \mu \! \left (h\cQ^{n_k} f_\bfv\right ) .
\eeqn
By Theorem~\ref{thm:bounds}, the last limit vanishes. In other words, there exists $g\in L^2(\mu)$ such that 
\beq\label{eq:coboundary_Q}
f_\bfv = g-\cQ g.
\eeq

\medskip
\noindent{\it Claim.} In fact,
\beq\label{eq:coboundary}
f_\bfv(x) = g(x) - g(T_{\omega_1}x)
\eeq
almost surely.

\medskip
\noindent Accepting the Claim for now, pick $\omega_1$ so that the above identity holds for almost every $x$. Standard Livschitz (Liv{\v{s}}ic) rigidity theory then shows that $g$ is H\"older continuous~\cite{Livsic_1971,Livsic_1972}. In particular,~\eqref{eq:coboundary} holds for all~$x$.
This proves the lemma in one direction. 

\medskip
To prove the lemma in the other direction, suppose~\eqref{eq:coboundary} holds almost surely for some nonzero vector $\bfv$ and some H\"older continuous $g$. Then $S_n = g - g\circ\pi\circ\Phi^{n}$ holds $P^\mu$-almost-everywhere, and
\beqn
E^\mu\!\left(S_n^2 \right) = \| g-g\circ\Phi^n \|_{L^2(P^\mu)}^2 \leq 4\| g \|_{L^2(\mu)}^2 .
\eeqn
Combining the bound with \eqref{eq:cov_new_bound2} we get $\bfv^\rT  \bfSigma^2 \bfv \leq n^{-1}\bigl(C|\bfv|^2 + 4\| g \|^2_{L^2(\mu)}\bigr)$ for all $n\geq 1$, which is only possible if $\bfSigma^2$ is degenerate in the direction of $\bfv$. 

\medskip
It thus remains to prove the earlier Claim. To that end, we define
\beqn
G_k = g\circ\pi\circ\Phi^k
\eeqn
and
\beqn
M_n = \sum_{k=0}^{n-1} (X_k - G_k + G_{k+1}) = S_n - G_0 + G_n.
\eeqn
Since $(S_n)_{n\ge 1}$ is uniformly bounded in $L^2(P^\mu)$, so is $(M_n)_{n\ge 1}$. Since~\eqref{eq:coboundary_Q} holds, the latter sequence is also a martingale adapted to the filtration $(\mathfrak{F}_n)_{n\ge 1}$ where $\mathfrak{F}_n$ is the sigma-algebra generated by the random variables $x,\omega_1,\dots,\omega_n$. Therefore,
\beqn
E^\mu(M_n^2) = \sum_{k=0}^{n-1} E^\mu\!\left((X_k - G_k + G_{k+1})^2\right) = n E^\mu\!\left((X_0 - G_0 + G_1)^2\right).
\eeqn
Combining these two facts, it follows that $X_0 - G_0 + G_1$ vanishes almost surely. The claim is proved.

\medskip
This finishes the proof of Lemma~\ref{lem:degeneracy}.
\qed



\bigskip

\bibliography{Random_expanding}{}

\begin{thebibliography}{10}

\bibitem{Aiminio_etal_2013}
Romain Aiminio, Matthew Nicol, and Sandro Vaienti.
\newblock Annealed and quenched limit theorems for expanding random dynamical
  systems.
\newblock 2013.
\newblock Preprint.

\bibitem{AyyerLiveraniStenlund_2009}
Arvind Ayyer, Carlangelo Liverani, and Mikko Stenlund.
\newblock Quenched {CLT} for random toral automorphism.
\newblock {\em Discrete Contin. Dyn. Syst.}, 24(2):331--348, 2009.
\newblock Available from: \url{http://dx.doi.org/10.3934/dcds.2009.24.331},
  \href {http://dx.doi.org/10.3934/dcds.2009.24.331}
  {\path{doi:10.3934/dcds.2009.24.331}}.

\bibitem{Billingsley_convergence}
Patrick Billingsley.
\newblock {\em Convergence of probability measures}.
\newblock Wiley Series in Probability and Statistics: Probability and
  Statistics. John Wiley \& Sons Inc., New York, second edition, 1999.
\newblock A Wiley-Interscience Publication.
\newblock Available from: \url{http://dx.doi.org/10.1002/9780470316962}, \href
  {http://dx.doi.org/10.1002/9780470316962} {\path{doi:10.1002/9780470316962}}.

\bibitem{BressaudLiverani}
Xavier Bressaud and Carlangelo Liverani.
\newblock Anosov diffeomorphisms and coupling.
\newblock {\em Ergodic Theory Dynam. Systems}, 22(1):129--152, 2002.
\newblock Available from: \url{http://dx.doi.org/10.1017/S0143385702000056},
  \href {http://dx.doi.org/10.1017/S0143385702000056}
  {\path{doi:10.1017/S0143385702000056}}.

\bibitem{Buzzi_1999}
J{\'e}r{\^o}me Buzzi.
\newblock Exponential decay of correlations for random {L}asota--{Y}orke maps.
\newblock {\em Comm. Math. Phys.}, 208(1):25--54, 1999.
\newblock Available from: \url{http://dx.doi.org/10.1007/s002200050746}, \href
  {http://dx.doi.org/10.1007/s002200050746} {\path{doi:10.1007/s002200050746}}.

\bibitem{Buzzi_2000}
J{\'e}r{\^o}me Buzzi.
\newblock Absolutely continuous {S}.{R}.{B}. measures for random
  {L}asota--{Y}orke maps.
\newblock {\em Trans. Amer. Math. Soc.}, 352(7):3289--3303, 2000.
\newblock Available from:
  \url{http://dx.doi.org/10.1090/S0002-9947-00-02607-6}, \href
  {http://dx.doi.org/10.1090/S0002-9947-00-02607-6}
  {\path{doi:10.1090/S0002-9947-00-02607-6}}.

\bibitem{Chernov_2006}
Nikolai Chernov.
\newblock Advanced statistical properties of dispersing billiards.
\newblock {\em J. Stat. Phys.}, 122(6):1061--1094, 2006.
\newblock Available from: \url{http://dx.doi.org/10.1007/s10955-006-9036-8},
  \href {http://dx.doi.org/10.1007/s10955-006-9036-8}
  {\path{doi:10.1007/s10955-006-9036-8}}.

\bibitem{ChernovMarkarian_2006}
Nikolai Chernov and Roberto Markarian.
\newblock {\em Chaotic billiards}, volume 127 of {\em Mathematical Surveys and
  Monographs}.
\newblock American Mathematical Society, Providence, RI, 2006.

\bibitem{Froyland_etal_2012}
Gary Froyland, Cecilia Gonz\'alez-Tokman, and Anthony Quas.
\newblock Stability and approximation of random invariant densities for
  {L}asota--{Y}orke map cocycles.
\newblock 2012.
\newblock Preprint.
\newblock Available from: \url{http://arxiv.org/abs/1212.2247}.

\bibitem{Gouezel_2010}
S{\'e}bastien Gou{\"e}zel.
\newblock Almost sure invariance principle for dynamical systems by spectral
  methods.
\newblock {\em Ann. Probab.}, 38(4):1639--1671, 2010.
\newblock Available from: \url{http://dx.doi.org/10.1214/10-AOP525}, \href
  {http://dx.doi.org/10.1214/10-AOP525} {\path{doi:10.1214/10-AOP525}}.

\bibitem{KhaninKifer_1996}
Konstantin Khanin and Yuri Kifer.
\newblock Thermodynamic formalism for random transformations and statistical
  mechanics.
\newblock In {\em Sina\u\i's {M}oscow {S}eminar on {D}ynamical {S}ystems},
  volume 171 of {\em Amer. Math. Soc. Transl. Ser. 2}, pages 107--140. Amer.
  Math. Soc., Providence, RI, 1996.

\bibitem{Kifer_2008}
Yuri Kifer.
\newblock Thermodynamic formalism for random transformations revisited.
\newblock {\em Stoch. Dyn.}, 8(1):77--102, 2008.
\newblock Available from: \url{http://dx.doi.org/10.1142/S0219493708002238},
  \href {http://dx.doi.org/10.1142/S0219493708002238}
  {\path{doi:10.1142/S0219493708002238}}.

\bibitem{LaceyPhilipp_1990}
Michael~T. Lacey and Walter Philipp.
\newblock A note on the almost sure central limit theorem.
\newblock {\em Statist. Probab. Lett.}, 9(3):201--205, 1990.
\newblock Available from: \url{http://dx.doi.org/10.1016/0167-7152(90)90056-D},
  \href {http://dx.doi.org/10.1016/0167-7152(90)90056-D}
  {\path{doi:10.1016/0167-7152(90)90056-D}}.

\bibitem{Livsic_1971}
A.~N. Liv{\v{s}}ic.
\newblock Certain properties of the homology of {$Y$}-systems.
\newblock {\em Mat. Zametki}, 10:555--564, 1971.

\bibitem{Livsic_1972}
A.~N. Liv{\v{s}}ic.
\newblock Cohomology of dynamical systems.
\newblock {\em Izv. Akad. Nauk SSSR Ser. Mat.}, 36:1296--1320, 1972.

\bibitem{Morita_1985}
Takehiko Morita.
\newblock Random iteration of one-dimensional transformations.
\newblock {\em Osaka J. Math.}, 22(3):489--518, 1985.
\newblock Available from:
  \url{http://projecteuclid.org/getRecord?id=euclid.ojm/1200778532}.

\bibitem{Ohno_1983}
Taijiro Ohno.
\newblock Asymptotic behaviors of dynamical systems with random parameters.
\newblock {\em Publ. Res. Inst. Math. Sci.}, 19(1):83--98, 1983.
\newblock Available from: \url{http://dx.doi.org/10.2977/prims/1195182976},
  \href {http://dx.doi.org/10.2977/prims/1195182976}
  {\path{doi:10.2977/prims/1195182976}}.

\bibitem{OttStenlundYoung}
William Ott, Mikko Stenlund, and Lai-Sang Young.
\newblock Memory loss for time-dependent dynamical systems.
\newblock {\em Mathematical Research Letters}, 16(3):463--475, 2009.
\newblock Available from:
  \url{http://intlpress.com/site/pub/pages/journals/items/mrl/content/vols/001%
6/0003/00020435/index.html}.

\bibitem{Pelikan_1984}
S.~Pelikan.
\newblock Invariant densities for random maps of the interval.
\newblock {\em Trans. Amer. Math. Soc.}, 281(2):813--825, 1984.
\newblock Available from: \url{http://dx.doi.org/10.2307/2000087}, \href
  {http://dx.doi.org/10.2307/2000087} {\path{doi:10.2307/2000087}}.

\bibitem{Pene_2005}
Fran{\c{c}}oise P{\`e}ne.
\newblock Rate of convergence in the multidimensional central limit theorem for
  stationary processes. {A}pplication to the {K}nudsen gas and to the {S}inai
  billiard.
\newblock {\em Ann. Appl. Probab.}, 15(4):2331--2392, 2005.
\newblock Available from: \url{http://dx.doi.org/10.1214/105051605000000476},
  \href {http://dx.doi.org/10.1214/105051605000000476}
  {\path{doi:10.1214/105051605000000476}}.

\bibitem{PhilippStout_1975}
Walter Philipp and William Stout.
\newblock Almost sure invariance principles for partial sums of weakly
  dependent random variables.
\newblock {\em Mem. Amer. Math. Soc. 2}, (issue 2, 161):iv+140, 1975.

\bibitem{Stenlund_2010}
Mikko Stenlund.
\newblock A strong pair correlation bound implies the {CLT} for {S}inai
  billiards.
\newblock {\em J. Stat. Phys.}, 140(1):154--169, 2010.
\newblock Available from: \url{http://dx.doi.org/10.1007/s10955-010-9987-7},
  \href {http://dx.doi.org/10.1007/s10955-010-9987-7}
  {\path{doi:10.1007/s10955-010-9987-7}}.

\bibitem{Stenlund_2011}
Mikko Stenlund.
\newblock Non-stationary compositions of {A}nosov diffeomorphisms.
\newblock {\em Nonlinearity}, 24:2991--3018, 2011.
\newblock \href {http://dx.doi.org/doi:10.1088/0951-7715/24/10/016}
  {\path{doi:doi:10.1088/0951-7715/24/10/016}}.

\bibitem{Stenlund_2012}
Mikko Stenlund.
\newblock A vector-valued almost sure invariance principle for {S}inai
  billiards with random scatterers.
\newblock 2012.
\newblock Preprint.
\newblock Available from: \url{http://arxiv.org/abs/1210.0902}.

\bibitem{StenlundYoungZhang_2012}
Mikko Stenlund, Lai-Sang Young, and Hongkun Zhang.
\newblock Sinai billiards with moving scatterers.
\newblock To appear in Communications in Mathematics.
\newblock Available from: \url{http://arxiv.org/abs/1210.0011}.

\bibitem{Strassen_1964}
Volker Strassen.
\newblock An invariance principle for the law of the iterated logarithm.
\newblock {\em Z. Wahrscheinlichkeitstheorie und Verw. Gebiete}, 3:211--226
  (1964), 1964.

\bibitem{Sulku_2012}
Henri Sulku.
\newblock Explicit correlation bounds for expanding circle maps using the
  coupling method.
\newblock 2012.
\newblock Unpublished manuscript. Based on the author's Bachelor's thesis,
  University of Helsinki.
\newblock Available from: \url{http://arxiv.org/abs/1211.2438}.

\bibitem{Tumel_2012}
Filiz T\"umel.
\newblock Random walks on a lattice with deterministic local dynamics.
\newblock 2012.
\newblock PhD Thesis, University of Houston.
\newblock Available from:
  \url{http://www.math.uh.edu/dynamics/DynSystGroup/Theses/2012-Filiz-Tumel.pd%
f}.

\bibitem{Vervaat_1979}
Wim Vervaat.
\newblock On a stochastic difference equation and a representation of
  nonnegative infinitely divisible random variables.
\newblock {\em Adv. in Appl. Probab.}, 11(4):750--783, 1979.
\newblock Available from: \url{http://dx.doi.org/10.2307/1426858}, \href
  {http://dx.doi.org/10.2307/1426858} {\path{doi:10.2307/1426858}}.

\bibitem{Young_1999}
Lai-Sang Young.
\newblock Recurrence times and rates of mixing.
\newblock {\em Israel J. Math.}, 110:153--188, 1999.
\newblock Available from: \url{http://dx.doi.org/10.1007/BF02808180}, \href
  {http://dx.doi.org/10.1007/BF02808180} {\path{doi:10.1007/BF02808180}}.

\end{thebibliography}
\bibliographystyle{plainurl}


\vspace*{\fill}


\end{document}